\documentclass[11pt,a4paper]{article}

\usepackage{jingfu}
\usepackage{comment}
\usepackage{graphicx}
\usepackage{amsmath}
\usepackage{amsfonts}
\usepackage{amssymb}
\usepackage{enumerate}
\usepackage{lscape}
\usepackage{longtable}
\usepackage{rotating}
\usepackage{color}
\usepackage{url}
\usepackage{rotating}
\usepackage{algpseudocode}
\usepackage[ruled]{algorithm}
\usepackage{pgf,tikz}
\usepackage{subfig}
\usepackage[titletoc]{appendix}
\usepackage{wrapfig}
\usepackage{adjustbox}
\usepackage{mathrsfs}
\usepackage{pgfplots}
\usetikzlibrary{arrows}
\usepackage{color}

\usepackage{bbm}

\usepackage{booktabs}
\usepackage{multirow}
\usepackage{mathtools}
\usepackage{makecell}

\usepackage[sort,numbers]{natbib}
\usepackage{eqparbox}%

\numberwithin{equation}{section}%

\usepackage[symbol]{footmisc}%

\newtheorem{proposition}{Proposition}[section]

\newtheorem{definition}{Definition}

\newtheorem{remark}{Remark}[section]

\newtheorem{assumption}{Assumption}[section]

\newenvironment{proof}[1][Proof]{\noindent \textbf{#1.} }{\hfill$\Box$\par\medskip}

\newcommand{\beqn}[1]{\begin{equation}\label{#1}}
\newcommand{\eeqn}{\end{equation}}

\definecolor{darkgreen}{rgb}{0,0.6,0}
\definecolor{aau2}{rgb}{0.0, 0.5, 0.69}
\definecolor{aau3}{rgb}{0.0, 0.53, 0.74}
\definecolor{aau4}{rgb}{0.0, 0.48, 0.65}
\definecolor{aau5}{rgb}{0.0, 0.45, 0.73}
\definecolor{rsap}{RGB}{130, 36, 51}
\definecolor{gsap}{RGB}{112, 164, 137}

\definecolor{tud}{rgb}{0.43,0.73,0.11}
\definecolor{verde}{rgb}{0.33,0.53,0.11}

\definecolor{ttffqq}{rgb}{0.0, 0.48, 0.65} 
\definecolor{ffqqqq}{rgb}{0.0, 0.5, 0.69} 

\usetikzlibrary{arrows}

\tikzstyle{decision} = [diamond, draw, fill=blue!20,
text width=4.5em, text badly centered, node distance=3cm, inner sep=0pt]
\tikzstyle{block} = [rectangle, draw, fill=blue!20,
text centered, rounded corners, minimum height=4em]
\tikzstyle{line} = [draw, -latex']
\tikzstyle{cloud} = [draw, ellipse,fill=red!20, node distance=3cm,
minimum height=2em]
\tikzstyle{cloud2} = [draw, ellipse,fill=green!20, node distance=3cm,
minimum height=2em]
\pgfplotsset{compat=1.18}
\begin{document}
	
\title{Non-smooth stochastic gradient descent \\ using smoothing functions}

	\author{
		T. Giovannelli\thanks{Department of Mechanical, Materials, and Industrial Engineering, University of Cincinnati, Cincinnati, OH 45221, USA ({\tt giovanto@ucmail.uc.edu}).}
		\and
		J. Tan\thanks{Department of Industrial and Systems Engineering, Lehigh University, Bethlehem, PA 18015-1582, USA ({\tt jit423@lehigh.edu}).}
		\and
		L. N. Vicente\thanks{Department of Industrial and Systems Engineering, Lehigh University, Bethlehem, PA 18015-1582, USA ({\tt lnv@lehigh.edu}).}
	}
	
	\maketitle

\section*{Abstract}
In this paper, we address stochastic optimization problems involving a composition of a non-smooth outer function and a smooth inner function, a formulation frequently encountered in machine learning and operations research. To deal with the non-differentiability of the outer function, we approximate the original non-smooth function using smoothing functions, which are continuously differentiable and approach the original function as a smoothing parameter goes to zero (at the price of increasingly higher Lipschitz constants). The proposed smoothing stochastic gradient method iteratively drives the smoothing parameter to zero at a designated rate. We establish convergence guarantees under strongly convex, convex, and non-convex settings, proving convergence rates that match known results for non-smooth stochastic compositional optimization. In particular, for convex objectives, smoothing stochastic gradient achieves a~$1/T^{1/4}$ rate in terms of the number of stochastic gradient evaluations. We further show how general compositional and finite-sum compositional problems (widely used frameworks in large-scale machine learning and risk-averse optimization) fit the
assumptions needed for the rates (unbiased gradient estimates, bounded second moments, and accurate smoothing errors). 
We present preliminary numerical results indicating that smoothing stochastic gradient descent can be competitive for certain classes of problems.

\section{Introduction} 
Stochastic compositional optimization problems, in which the objective function involves a compositional structure of an inner and an outer function, have found increasing applications in machine learning and risk-averse optimization (see \cite{chen2021solving,shapiro2021lectures}). In this paper, we will consider the following two stochastic compositional problems  
    \begin{equation}\label{prob:1}
	\begin{split}
	\min_{x \in \Rmbb^n} ~~ & {\phi}(\mathbb{E}[\psi(x,\xi)]), 
    \end{split}
	\end{equation}
and 
    \begin{equation}\label{prob:2}
	\begin{split}
	\min_{x \in \Rmbb^n} ~~ &\mathbb{E} [{\phi}(\psi(x,\xi))], 
    \end{split}
	\end{equation}
where~$x$ is the vector of decision variables,~$\xi$ is a random vector,~$\Embb$ is taken with respect to the probability distribution of $\xi$,~$\psi$ is a differentiable function, and~$\phi$ is a function which is not necessarily smooth (meaning that the function~$\phi$ may be non-differentiable).

Both types of problems naturally occur in fields like reinforcement learning \cite{laguel2021superquantiles,dann2014policy}, risk management \cite{ahmed2007coherent,liu2022solving}, meta-learning \cite{finn2017model}, and distributionally robust
optimization \cite{li2023tilted}. An example of~${\phi}(\mathbb{E}[\psi(x,\xi)])$ involves super and sub-quantiles and functions (see, e.g.,~\cite{giovannelli2026stochastic}) that can be written as
    \[
    \max~\{ 0, \mathbb{E}[\psi(x,\xi)]) \}.
    \]
An instance of~$\min_{x \in \Rmbb^n}~\mathbb{E} [{\phi}(\psi(x,\xi))]$ is expected risk minimization, which is widely adopted in machine learning applications. A specific example is given as follows
\[
\min_{x \in \Rmbb^n}~~ \mathbb{E} [|x^\top u + b - v|],
\]
where the outer function is the $L_1$ loss, the inner function is linear, and $\xi = (u, v)$. We will discuss the formulations~\eqref{prob:1} and~\eqref{prob:2} in more detail in Sections~\ref{section: compositional} and~\ref{section: finite-sum}.




Due to the existence of non-differentiable points of the non-smooth outer function~$\phi$, traditional gradient-based methods cannot be directly applied to problems~\eqref{prob:1} and~\eqref{prob:2}.
To deal with the non-smoothness of~$\phi$, we will use a smoothing function (see Definition~\ref{def: smoothing funtion} below). Smoothing functions have been introduced in~\cite{chen1996class,zhang2009smoothing}. 

\begin{definition} \label{def: smoothing funtion}
    Let~$g:\Rmbb^n \rightarrow \Rmbb$ be a locally Lipschitz continuous function. We call~$\gtilde:\Rmbb^n\times[0,+\infty)\rightarrow\Rmbb$ a smoothing function of~$g$ if, for any~$\mu\in(0,+\infty)$,~$\gtilde(\cdot,\mu)$ is continuously differentiable in~$\Rmbb^n$ and, for any~$x\in\Rmbb^n$,
    \begin{equation} \label{eq: smoothing function}
        \lim_{z\rightarrow x,\mu\downarrow0}~\gtilde(z,\mu) \;=\; g(x).
    \end{equation}
\end{definition}

Notice that in this paper, for the convergence analysis, we adopt a general viewpoint in which the composite objective is treated as a non-smooth function covering both cases~\eqref{prob:1} and~\eqref{prob:2}. In the concrete settings considered in Sections~4 and~5, however, we take advantage of the problem structure and smooth the outer function~$\phi$, which will yield more explicit smoothing models and stochastic gradient estimators.

In Section~\ref{section: smoothing function}, we will provide an introduction to the essential properties of smoothing functions.
We will also see in Sections~\ref{section: compositional} and~\ref{section: finite-sum} what conditions~$\tilde{\phi}$ and~$\psi$ need to satisfy to ensure that~$\Embb[\tilde{\phi}(\psi(\cdot,\xi),\mu)]$ and $\tilde{\phi}(\Embb[\psi(\cdot,\xi)],\mu)$ are smoothing functions of~$\Embb[\phi(\psi(\cdot,\xi))]$ and~$\phi(\Embb[\psi(\cdot,\xi)])$, respectively (satisfying appropriate accuracy and convexity requirements). 
We will assume that we can compute stochastic smoothing gradients for both $\Embb[\tilde{\phi}(\psi(\cdot,\xi),\mu)]$ and $\tilde{\phi}(\Embb[\psi(x,\xi)],\mu)$ (the explicit forms will be shown in Sections~\ref{section: compositional} and~\ref{section: finite-sum}). 
As the smoothing parameter approaches zero, the smoothing function will approach the original function but at the price of larger gradient Lipschitz constants.

\subsection{Short review of convergence rates for gradient descent}\label{review}

Gradient descent (GD) and stochastic gradient descent (SGD) are widely used in optimization and machine learning problems due to their simplicity, efficiency, and effectiveness, particularly in terms of low iteration cost and memory storage requirements. The first development of GD can be traced back to~1847. Cauchy introduced a method of iterative improvement, laying out the fundamental idea of moving toward function minima by repeatedly taking steps proportional to the negative gradient. 
SGD arose in the context of the statistical approximation method, which was first introduced by Robbins and Monro \cite{robbins1951stochastic} in 1951. Over the subsequent decades, SGD gained widespread popularity in large-scale problems~\cite{bottou2018optimization}, owing to its efficiency and scalability in handling large datasets and high-dimensional parameter spaces.

Modern convergence theory traces back to Polyak’s seminal work on GD for $L$-smooth objectives in the 1960s \cite{polyak1963gradient,polyak1964some}. For a convex and $L$-smooth function $f$, taking a fixed stepsize~$1/L$ yields the classical sub-linear rate $\mathcal{O}(1/T)$ \cite{nesterov2013introductory}. When $f$ is additionally $\mu$-strongly convex, the same scheme converges at the linear rate $\mathcal{O}((1-\mu/L)^{T})$ \cite{polyak1964some}. In the non-smooth convex setting, replacing gradients with subgradients and using the diminishing stepsizes $\gamma_t=\Theta(1/\sqrt{t})$ attains the optimal rate $\mathcal{O}(1/\sqrt{T})$ \cite{nemirovskij1983problem,shor2012minimization}. Furthermore, for non-smooth and $\mu$-strongly convex objectives, choosing $\gamma_t=\Theta(1/t)$ tightens the bound to~$\mathcal{O}(1/T)$~\cite{polyak1987introduction}.

For SGD, analogous complexity results were developed. On a convex and $L$-smooth objective with unbiased stochastic gradients of bounded second moments, SGD with diminishing stepsizes~$\gamma_t = \Theta(1/\sqrt{t})$ achieves the optimal sub-linear rate $\mathcal{O}(1/\sqrt{T})$ \cite{nemirovski2009robust}. If the objective is additionally $\mu$-strongly convex, choosing $\gamma_t = \Theta(1/t)$ tightens the rate to~$\mathcal{O}(1/T)$ \cite{moulines2011non}. If smoothness is absent, in the convex setting, replacing gradients with subgradients while retaining $\gamma_t = \Theta(1/\sqrt{t})$ yields the classical~$\mathcal{O}(1/\sqrt{T})$ bound \cite{nemirovskij1983problem,shor2012minimization}. Finally, for non-smooth and $\mu$-strongly convex objectives, a subgradient scheme with~$\gamma_t = \Theta(1/t)$ attains the optimal~$\mathcal{O}(1/T)$ convergence rate~\cite{nemirovski2009robust,moulines2011non}.

We identify two main categories of optimization problems: the additive form~$f(x)=g(x)+h(x)$, which leads to a composite optimization problem, and the inner–outer (nested) form~$f(x)=\phi(\psi(x))$, which leads to a compositional optimization problem, each bringing its own algorithmic challenges and solution methods.
Considering the additive composite formulation~$f(x)=g(x)+h(x)$, suppose that~$g$ is smooth and~$L$-Lipschitz differentiable, and~$h$ is proximable and potentially non-smooth. Proximal-gradient methods have been extensively studied and widely used for solving this type of problem. Notable methods include ISTA~\cite{daubechies2004iterative} and its accelerated variant FISTA~\cite{beck2009fast}. Stochastic variants, such as FOBOS~\cite{duchi2009efficient} and RDA~\cite{xiao2009dual}, have also been explored. Additionally, accelerated stochastic proximal methods like Prox-SVRG and SAGA~\cite{xiao2014proximal,defazio2014saga} have gained popularity due to their improved convergence rates.
In contrast to the well-studied additive composite form, research on stochastic compositional optimization, defined by its nested inner-outer structure, has begun to attract attention in recent years. An early contribution is due to Ermoliev~\cite{ermoliev1976stochastic}, who proposed a two-timescale stochastic approximation method for optimizing the composition of two expectation‐valued functions; see Section~6.7 of~\cite{ermoliev1988numerical} for a brief related discussion.  A more recent advance came from Wang et al.~\cite{wang2017stochastic}, who formally introduced the stochastic compositional gradient descent (SCGD) method to solve problems of the form~$f(x)=\mathbb{E}[\phi(\mathbb{E}[\psi(x,\xi_1)],\xi_2)]$, where the outer expectation is taken with respect to the probability distribution of~$\xi_2$, the inner one is taken with respect to the probability distribution of~$\xi_1$, and the inner function may be non-smooth. The SCGD algorithm operates on two intertwined time-scales, employing an auxiliary sequence to estimate the inner expectation via an exponential moving average, while the outer iterate descends along a stochastic quasi-gradient formed from this estimation. For non-smooth objectives, SCGD achieves sample complexities of $\mathcal{O}(T^{-1/4})$ and $\mathcal{O}(T^{-2/3})$ in the convex and strongly convex cases, respectively. Compared to our method, SCGD primarily focuses on handling the non-smoothness of the inner function, whereas our approach specifically targets scenarios where the outer function is potentially non-smooth. Additionally, our method employs an iterative smoothing parameter scheme, ensuring controlled smoothing accuracy throughout the optimization process. 

Building on SCGD, subsequent work further developed the nested type stochastic compositional optimization problems (and their variants). Wang et al.~\cite{lian2017finite} integrated SVRG into SCGD for finite-sum compositional objectives~$f(x)=\frac{1}{N}\sum_{i=1}^N\phi_i\left(\frac{1}{M}\sum_{j=1}^M\psi_j(x)\right)$ (where~$\phi_i$ and~$\psi_j$ denote the values of~$\phi$ and~$\psi$ for specific realizations of their respective random vectors) and proved a linear convergence rate under strong convexity. For smooth but non-convex stochastic compositional objectives~$f(x)=\mathbb{E}[\phi(\mathbb{E}[\psi(x,\xi_1)],\xi_2)]$, advances have been made by~\cite{yuan2019stochastic,liu2021variance}.
In addition, several works extend proximal-gradient methods and their variants~\cite{wang2017accelerating,huo2018accelerated}, employ ADMM~\cite{yu2017fast}, use SARAH~\cite{zhang2019stochastic}, or adopt SAGA~\cite{zhang2019composite} to address stochastic compositional problems of the form $f(x)=\mathbb{E}[\phi(\mathbb{E}[\psi(x,\xi_1)],\xi_2)]+h(x)$, where both the inner and outer functions are smooth and a non-smooth regularization term $h$ is incorporated. 
In addition, for structured non-convex composite models of the form~$f(g(x))+h(x)$, where~$f$ and~$h$ are convex, possibly non-smooth, and~$g$ is a smooth vector mapping represented either as a finite average or as an expectation, 
it was shown in~\cite{zhang2022stochastic} that stochastic variance-reduced prox-linear methods achieve sample-complexity guarantees for finding an~$\varepsilon$-stationary point.

\subsection{Gradient descent using smoothing techniques}

Besides the methods discussed in Section~\ref{review}, another prominent approach for handling non-smooth objectives is the use of smoothing functions, which approximate a non-smooth function by a smooth surrogate that converges to the original function as the smoothing parameter diminishes.
Nesterov proposes a smoothing framework~\cite{nesterov2005smooth} (infimal convolution smoothing) and demonstrates that structured non-smooth convex problems admit an~$\Ocal(1/\varepsilon)$ first-order complexity bound for obtaining an~$\varepsilon$-accurate solution, improving upon the classical subgradient rate of~$\Ocal(1/\varepsilon^2)$. Alternative smoothing techniques, such as randomized smoothing~\cite{duchi2012randomized} and Moreau envelope smoothing~\cite{moreau1965proximite}, have been extensively explored in the literature. In derivative-free settings, notable smoothing approaches include Gaussian smoothing~\cite{nesterov2017random} and direct search using smoothing functions~\cite{garmanjani2013smoothing}.

For \textit{additive} composite optimization problems of the form~$f(x) = g(x) + h(x)$, where~$h(x)$ may be non-smooth, various smoothing techniques have been explored in the literature. 
Randomized smoothing~\cite{duchi2012randomized} was introduced to address this type of problem. Accelerated proximal gradient methods enhanced by Nesterov's smoothing technique have also been explored~\cite{lin2014smoothing, wang2022stochastic}. Additionally, an approach using the Moreau envelope along with a variable sample-size accelerated proximal scheme to manage non-smoothness has been proposed in~\cite{jalilzadeh2022smoothed}. For weakly convex non-smooth problems of the form~$f(x)=g(x)+r(x)$, where~$g$ is weakly convex and~$r$ is a closed convex function with a computable proximal map, Davis and Drusvyatskiy~\cite{davis2018stochastic} showed that a proximal stochastic subgradient method drives the gradient of the Moreau envelope to zero at a rate of~$\Ocal(T^{-1/4})$.
By contrast, to the best of our knowledge, existing studies on stochastic compositional problems with an inner–outer structure do not employ a smoothing gradient method that iteratively updates the smoothing parameter, as proposed in this paper.

\subsection{Compositional case using smoothing functions---our contribution}

The main goal of this paper is to develop a smoothing stochastic gradient (SSG) method using the gradients of smoothing functions, where the smoothing parameter is decreased at a controlled rate. Our motivation is drawn from compositional problems~\eqref{prob:1} and~\eqref{prob:2}, but we will abstract from these two settings and develop the SSG method and its convergence theory for a general scenario (see formulation~\eqref{problem: f(x)} in Section~\ref{section: general}) for which there is an accurate smoothing function, and for which there is the possibility of drawing unbiased smoothing stochastic gradients satisfying a bounded second order moment dependent on the Lipschitz constant of the smoothing function.
We will investigate the convergence of the SSG method for this general scenario under strongly convex, convex, and non-convex regimes (see Sections~\ref{subsection: convex}--\ref{subsection: non-convex} and Appendix~\ref{app:strongly convex}). Notably, for convex objectives, the SSG method attains a $1/T^{1/4}$ convergence rate in terms of stochastic gradient evaluations. 
We also show in Sections~\ref{section: compositional} and~\ref{section: finite-sum} that widely used frameworks such as compositional and finite-sum compositional problems
(which are instances of problems~\eqref{prob:1} and~\eqref{prob:2}, respectively) meet the required smoothing assumptions, including convexity, unbiased gradient estimates, bounded second moments, and controlled errors. The illustrative numerical experiments reported in Section~\ref{section: numerical} show that the SSG method can perform competitively for certain problem classes.

\section{Introduction to smoothing functions} \label{section: smoothing function}
In this section, we will present two main examples and their relevant properties of smoothing functions for the context of our paper. We specifically address widely used non-smooth functions in optimization, namely the~$L_1$ and the Hinge losses. At the end of the section, we mention one way of constructing smoothing functions through convolution with mollifiers, yielding smoothed functions that satisfy the key properties we outlined.

The $L_1$ norm is widely used in optimization, for instance in sparse optimization and compressed sensing (see \cite{donoho2006compressed,tibshirani1996regression}). Due to its non-smooth nature, it is advantageous to replace it with a smooth approximation. One way to smooth~$\lvert \cdot \rvert$ is introduced in \cite{chen2010smoothing}
\begin{equation} \label{example: l1}
\tilde{s}(t,\mu)\;=\;\int_{-\infty}^{+\infty} \big|t-\mu\tau\big|\,\rho(\tau)\,d\tau,    
\end{equation}
where $\rho:\mathbb{R}^n\to[0,+\infty)$ is a symmetric and piecewise continuous density function with a finite number of pieces, satisfying $\int_{-\infty}^{+\infty} |\tau|\,\rho(\tau)\,d\tau <+\infty$. 

The function~\eqref{example: l1} holds several important properties.  Firstly, $\tilde{s}(\cdot,\mu)$ is continuously differentiable on $\mathbb{R}$, with its derivative expressed as~$\tilde{s}'(t,\mu)=2\int_{0}^{t/\mu}\rho(\tau)d\tau$. As~$\mu$ approaches zero, the function converges uniformly to~$|t|$ on~$\mathbb{R}$, with the approximation error bounded by a multiple of the smoothing parameter~$\mu$. Moreover, this smoothing function satisfies the so-called gradient consistency, meaning that the limit points of~$\tilde{s}'(t,\mu)$ approach the Clarke subdifferential of~$|t|$, more specifically, it holds that~$\{ \lim_{t \to 0, \mu \downarrow 0} \tilde{s}'(t, \mu) \} = [-1,1]=\partial_{\lvert\cdot\rvert}(0)$. Finally, for any fixed~$\mu>0$, the derivative~$\tilde{s}'(t,\mu)$ is Lipschitz continuous with a constant that is proportional to~$1/\mu$.

An explicit example of~\eqref{example: l1} with a uniform density 
\[
\rho(\tau)\;=\;
\begin{cases}
\frac{1}{2}, & \tau\in\left[-\frac{1}{2},\frac{1}{2}\right],\\[1mm]
0, & \text{otherwise},
\end{cases}
\]
yields
\begin{equation} \label{function: |.|}
 \tilde{s}(t,\mu)\;=\;
\begin{cases}
\frac{t^2}{\mu}+\frac{\mu}{4}, & t\in\left[-\frac{\mu}{2},\frac{\mu}{2}\right],\\[1mm]
|t|, & \text{otherwise}.
\end{cases}   
\end{equation}

Besides the $L_1$ loss, the Hinge loss is also broadly used in machine learning and optimization applications. Similarly to the $L_1$ case, a particular smoothing function for the Hinge loss is
\begin{equation} \label{function: Hinge}
    \tilde{h}(t,\mu) \;=\; \begin{cases}
			1-t,                   & \quad t\leq 1 - \mu,    \\
            \frac{1}{4\mu}(1-t+\mu)^2,               & \quad 1-\mu<t\leq1 +\mu,   \\
            0,                            & \quad t >  1+\mu.
            \end{cases}
\end{equation}
This smoothing function holds the same properties as~\eqref{function: |.|}.

Generally speaking, there are various techniques to construct a smoothing function. For example, one can smooth the non-smooth function by convolution with mollifiers (see~\cite{rockafellar1998variational}). A parameterized family of measurable functions~$\{\psi_\mu : \mathbb{R}^n \to [0,+\infty),~\mu\in(0,\infty)\}$ is called a mollifier (or mollifier sequence) if it satisfies $\int_{\mathbb{R}^n} \psi_\mu(z)\,dz = 1$ and if~$B_\mu = \{ z : \psi_\mu(z) > 0 \}$ is bounded and converges to $\{0\}$ as $\mu \downarrow 0$.
A smoothing function~$\ftilde$ for a non-smooth function~$f$ is then constructed via convolution with mollifiers as follows
\begin{equation} \label{eq: conv_with_mol}
\tilde{f}(x,\mu) = \int_{\mathbb{R}^n} f(x-z)\,\psi_\mu(z)\,dz = \int_{\mathbb{R}^n} f(z)\,\psi_\mu(x-z)\,dz.    
\end{equation}
This construction guarantees pointwise convergence from $\tilde{f}(x,\mu)$ to~$f(x)$ as~$\mu$ goes to $0$~(thus satisfies~\eqref{eq: smoothing function} in Definition~\ref{def: smoothing funtion}). If the mollifiers are continuous on~$\Rmbb^n$, then $\tilde{f}(\cdot,\mu)$ is continuously differentiable. Note that when~$f$ is locally Lipschitz continuous at a limit point~$x_*$, one always has~$\partial f(x_*)\subseteq\operatorname{co} G_{\ftilde}(x_*)$ (see~\cite{chen2012smoothing}), where~$\partial f(x_*)$ is the Clarke subdifferential,~$\operatorname{co}$ denotes the convex hull, and~$G_{\ftilde}(x_*)=\{\lim_{x\to x_*,\mu\downarrow0}\nabla\ftilde(x,\mu)\}$. Moreover, using construction~\eqref{eq: conv_with_mol}, when~$f$ is Lipschitz continuous near~$x_*$, the so-called gradient consistency property, which says that~$\partial f(x_*) = \operatorname{co}G_{\tilde{f}}(x_*)$, is satisfied. Other than convolution with mollifiers, smoothing functions can also be constructed by convolution with probability density functions (see~\cite{rockafellar1998variational}).

\section{The smoothing stochastic gradient method} \label{section: general}
\subsection{Algorithm and assumptions}
We will be stating and analyzing a smoothing stochastic gradient method to solve the expected risk problem written in the general form
\begin{equation} \label{problem: f(x)}
    \min_{x\in \Rmbb^n}~f(x)~~(f~\text{involves}~\xi,~\psi~\text{and}~\phi),
\end{equation}
where again~$x$ is the vector of decision variable,~$\xi$ is a random vector,~$\psi$ is a differentiable function, and~$\phi$ is a function which is not necessarily smooth. Problem~\eqref{problem: f(x)} is meant to include cases~\eqref{prob:1} and~\eqref{prob:2} as two particular instances. We consider that a smoothing function~$\ftilde(x,\mu)$ is available for~$f(x)$ involving a smoothing function~$\tilde{\phi}$ of~$\phi$ and also including the randomness given by~$\xi$. We also assume that one can draw gradient estimates~$\nabla_x\ftilde(x,\xi,\mu)$ for the smoothing function~$\ftilde(x,\mu)$. A smoothing stochastic gradient method can then be stated in Algorithm~$\ref{algorithm: 1}$ and it only differs from the traditional stochastic gradient algorithm in the update of the smoothing parameter~$\mu_k$. 
\begin{algorithm}[H] 
	\caption{Smoothing stochastic gradient (\textbf{SSG}) method}\label{algorithm: 1}
	\begin{algorithmic}[1]
		\medskip
		\item[] {\bf Input:} $x_1$,~$\{\alpha_k\}$, and~$\{\mu_k\}$.
		\medskip
		\item[] {\bf For~$k = 1, 2, \ldots$ \bf do}
		\item[] \quad {\bf Step 1.} 
        Generate a realization~$\xi_k$ of the random variable~$\xi$.\;Then compute~$\nabla_x\ftilde(x_k,\xi_k,\mu_k)$.
            \nonumber
            \item[] \quad {\bf Step 2.}
        Update iterate~$
x_{k+1}=x_k-\alpha_k \nabla_x \tilde f(x_k,\xi_k,\mu_k)
$.
            \nonumber
		\item[] \quad {\bf Step 3.} 
        Update smoothing parameter~$\mu_{k+1}\leq\mu_k$.
		\item[] {\bf End do}
    	\end{algorithmic}
\end{algorithm}

In this section, we will analyze the convergence of the SSG algorithm in the convex and non-convex regimes. The strongly convex case has less applicability and is left to the appendix of this paper. We first introduce the general assumptions required for the convergence of the SSG algorithm in all three regimes.

Before we move on, we want to note that we state the assumptions at the level of the smoothing objective~$\tilde f$ rather than the original non-smooth objective~$f$. By doing this, our analysis will include a vast class of functions~$f$ and corresponding smoothing functions~$\ftilde$ (such as the ones considered in Sections~4 and~5).


Assumption~\ref{assumption: unbiasedness} states the unbiasedness of the smoothing gradient~$\nabla_x \ftilde(x, \xi, \mu)$, meaning that its expectation over random vector~$ \xi $
coincides with~$\nabla_x \ftilde(x, \mu)$. This requirement is typically mild and will be shown to hold under reasonable conditions in Sections~\ref{section: compositional} and~\ref{section: finite-sum}. 
\bassumption [Unbiasedness of the smoothing gradient] \label{assumption: unbiasedness}
For all~$x\in \Rmbb^n$ and for all $\mu>0$, the smoothing gradient~$\nabla_x \ftilde(x, \xi, \mu)$ serves as an unbiased stochastic estimator of~$\nabla_x \ftilde(x, \mu)$, meaning that
\begin{equation*}
       \Embb_\xi[\nabla_x \ftilde(x,\xi,\mu)]\;=\;\nabla_x \ftilde(x,\mu).  
\end{equation*}
\eassumption

Assumption~\ref{assumption: bounded variance} requires that the second moment of the smoothing gradient is bounded by a constant of the same order as~$1/\mu^2$. This is again a mild requirement, and it will be shown to hold under suitable conditions for most smoothing functions in Sections~\ref{section: compositional} and~\ref{section: finite-sum}.
\bassumption[Bounded second moment of the smoothing gradient] \label{assumption: bounded variance}
For all~$x\in \Rmbb^n$ and for all~$\mu>0$, the smoothing gradient~$\nabla_x \ftilde(x,\xi,\mu)$ is bounded as follows 
   \begin{equation*}
        \Embb_{\xi} [\|\nabla_x \ftilde(x,\xi,\mu)\|^2]\;\leq\; \frac{1}{\mu^2}G^2,
   \end{equation*}
where~$G>0$ is a positive constant.
\eassumption

Lastly, Assumption~\ref{assumption: sc: diff} ensures that the smoothing function~$\ftilde(x,\mu)$ closely approximates the true function~$f(x)$, with their difference (accuracy) controlled by the smoothing parameter~$\mu$. Similarly as for Assumption~\ref{assumption: unbiasedness} and Assumption~\ref{assumption: bounded variance}, such a requirement is mild and will be shown to hold under appropriate conditions for most smoothing functions in Sections~\ref{section: compositional} and~\ref{section: finite-sum}.
\bassumption[Accuracy of the smoothing function] \label{assumption: sc: diff}
For all~$x\in \Rmbb^n$ and for all~$\mu>0$, the difference between the true function and the smoothing function is bounded by some constant times the smoothing parameter in the following way
\begin{equation} \label{eq: accuracy}
   |f(x)-\ftilde(x,\mu)|\;\leq\; C\mu, 
\end{equation}
for some $C>0$.
\eassumption

By now, we have outlined the general assumptions required for the SSG algorithm. Next, we will present additional specific assumptions for the convex and non-convex cases, and then develop the corresponding convergence results. 

\subsection{Rate in the convex case} \label{subsection: convex}
In this subsection, we examine the convex case. While this generally leads to milder convergence guarantees compared to the strongly convex setting, it applies to a wider range of problems. We begin by outlining the assumption specific to the convex case and then present the convergence analysis for the SSG algorithm under such a condition. Assumption~\ref{assumption: convex} states the convexity of the smoothing function.
\bassumption[Convexity of the smoothing function] \label{assumption: convex}
For all $x\in \Rmbb^n$ and for all~$\mu>0$, the smoothing function $\ftilde(x,\mu)$ is convex in~$x$. 
\eassumption

To establish the convergence results, we also need to consider an additional assumption regarding the boundness of the sequence of iterates.
\bassumption[Boundedness of the sequence of iterates] \label{assumption: boundness feasible region}
There exists a positive constant $M>0$ such that~$\|x_k-x_*\|\leq M$ for all $k$, where~$x_*$ is any minimizer of~$f$.
\eassumption
\begin{remark} \label{remark: big M}
In the convex case, Assumption~\ref{assumption: boundness feasible region} is introduced to control the iterate sequence. Although the SSG algorithm considered in this paper is not formulated in projected form, this assumption is natural when one has in mind a constrained formulation over a bounded convex set~$\Xcal$. Indeed, if the initial iterate~$x_1$ is chosen in~$\Xcal$ and each iterate is projected onto~$\Xcal$, then boundedness follows automatically. Therefore, in the following proof, we consider the projected update rule obtained by replacing Step~2 of Algorithm~\ref{algorithm: 1} with
$x_{k+1}=\Pi_{\Xcal}(x_k-\alpha_k \nabla_x \tilde f(x_k,\xi_k,\mu_k))$,
where~$\Pi_{\Xcal}$ denotes the Euclidean projection onto~$\Xcal$.
\end{remark}

Building upon the convexity assumption and considering specific rates for both the smoothing parameter and the stepsize in the SSG method, we now present the convergence results in the convex setting, which provides a weaker convergence rate compared to the strongly convex case given in the appendix.
\btheorem[Convergence rate in the convex case] \label{theorem: convex}
Under Assumptions \ref{assumption: unbiasedness}--\ref{assumption: boundness feasible region}, suppose that the SSG algorithm runs with the smoothing parameter $\mu_k=k^{-a}$ and stepsize $\alpha_k=k^{-b}$, where $a>0$ and $b>0$. Define $f^T_{\text{best}}=\min_{k=\{0,1,\ldots,T\}}f(x_k)$. Then for any minimizer $x_*$ of $f$, one has
\begin{equation} \label{rate: convex}
    \begin{aligned}
        \Embb[f^T_{\text{best}}]-f(x_*) &\;\leq\; \frac{2Cc_1}{T^a}+\frac{M^2}{T^{1-b}} + \frac{c_2G^2}{2} \frac{1}{T^{-2a+b}},
    \end{aligned}
\end{equation}
where~$c_1$ and~$c_2$ are some positive constants.

Moreover, one can show that the best convergence rate of \eqref{rate: convex} is of the order of~$1/T^{1/4}$.

\etheorem
\bproof
First, using the projected update rule of the SSG algorithm stated in Remark~\ref{remark: big M}, together with the non-expansiveness of the projection operator, we obtain the following
\begin{equation} \label{eq: convex: using update rule}
\begin{aligned}
&\mathbb{E}[\|x_{k+1}-x_*\|^2|x_k] \\
&= \mathbb{E}[\|\Pi_\Xcal(x_k-\alpha_k \nabla_x \tilde f(x_k,\xi_k,\mu_k))-x_*\|^2 | x_k] \\
&= \mathbb{E}[\|\Pi_\Xcal(x_k-\alpha_k \nabla_x \tilde f(x_k,\xi_k,\mu_k))-\Pi_\Xcal(x_*)\|^2 | x_k] \\
&\le \mathbb{E}[\|x_k-\alpha_k \nabla_x \tilde f(x_k,\xi_k,\mu_k)-x_*\|^2 | x_k] \\
&= \mathbb{E}[\|x_k-x_*\|^2
-2\alpha_k \langle \nabla_x \tilde f(x_k,\xi_k,\mu_k),x_k-x_*\rangle
+\alpha_k^2 \|\nabla_x \tilde f(x_k,\xi_k,\mu_k)\|^2 | x_k] \\
&= \|x_k-x_*\|^2
-2\alpha_k \,\mathbb{E}[\langle \nabla_x \tilde f(x_k,\xi_k,\mu_k),x_k-x_*\rangle | x_k]
+\alpha_k^2 \,\mathbb{E}[\|\nabla_x \tilde f(x_k,\xi_k,\mu_k)\|^2 | x_k] \\
&= \|x_k-x_*\|^2
-2\alpha_k\langle \nabla_x \tilde f(x_k,\mu_k),x_k-x_*\rangle
+\alpha_k^2 \,\mathbb{E}[\|\nabla_x \tilde f(x_k,\xi_k,\mu_k)\|^2 | x_k].
\end{aligned}
\end{equation}

Then, by the convexity assumption of $\ftilde(x,\mu)$ (Assumption \ref{assumption: convex}), we have
\begin{equation} \label{ieq: convexity}
    \langle \nabla_x \ftilde(x_k,\mu_k),x_k-x_*\rangle \geq  \ftilde(x_k,\mu_k)-\ftilde(x_*,\mu_k).
\end{equation}
By combining \eqref{eq: convex: using update rule} and \eqref{ieq: convexity}, and taking the total expectation on both sides, we get
\begin{equation*} \label{ieq: convex: combine}
    \begin{aligned}
        \Embb[\|x_{k+1}-x_*\|^2] &\leq \Embb[\|x_k-x_* \|^2]  - 2\alpha_k\Embb[\ftilde(x_k,\mu_k)-\ftilde(x_*,\mu_k)] + \alpha_k^2 \Embb[\|\nabla_x \ftilde(x_k,\xi_k,\mu_k)\|^2].
    \end{aligned}
\end{equation*}
Using the boundedness assumption of the smoothing gradient estimates (Assumption \ref{assumption: bounded variance}), and dividing both sides by $\alpha_k$, we obtain
\begin{equation*} \label{ieq: convex: boundness assumption}
    \begin{aligned}
     2\Embb[\ftilde(x_k,\mu_k)-\ftilde(x_*,\mu_k)] &\leq\frac{\Embb[\|x_k-x_* \|^2] - \Embb[\|x_{k+1}-x_*\|^2]}{\alpha_k}   +  \frac{\alpha_k}{\mu_k^2}G^2. 
    \end{aligned}
\end{equation*}

Then, by summing the inequality from $k=1$ to $T$ for some $T>0$ and using the boundedness assumption of the sequence of iterates (Assumption~\ref{assumption: boundness feasible region}), we obtain
\begin{equation} \label{ieq: convex: sum}
\begin{aligned}
    \sum_{k=1}^T 2\Embb[\ftilde(x_k,\mu_k)-\ftilde(x_*,\mu_k)] 
    &\leq \frac{1}{\alpha_1}\Embb[\|x_1-x_* \|^2] + \sum_{k=2}^T(\frac{1}{\alpha_k}-\frac{1}{\alpha_{k-1}}) \Embb[\|x_{k}-x_*\|^2]+  \sum_{k=1}^T\frac{\alpha_k}{\mu_k^2}G^2 \\
    &\leq \frac{2M^2}{\alpha_T} + \sum_{k=1}^T\frac{\alpha_k}{\mu_k^2}G^2.
\end{aligned}
\end{equation}
From inequalities \eqref{eq: accuracy} (for~$x_n$ and~$x_*$) and \eqref{ieq: convex: sum}, we have
\begin{equation} \label{ieq: convex: after sum}
    \begin{aligned}
        &\sum_{k=1}^T 2\Embb[f(x_k)-f(x_*)] \\
        &\leq -\sum_{k=1}^T 2\Embb[\ftilde(x_k,\mu_k)-f(x_k)] + \sum_{k=1}^T2\Embb[\ftilde(x_*,\mu_k)-f(x_*)]+\frac{2M^2}{\alpha_T} + \sum_{k=1}^T\frac{\alpha_k}{\mu_k^2}G^2 \\
        &\leq 4C\sum_{k=1}^T \mu_k+\frac{2M^2}{\alpha_T} + \sum_{k=1}^T\frac{\alpha_k}{\mu_k^2}G^2.
    \end{aligned}
\end{equation}

Now let us first consider $\sum_{k=1}^T\mu_k = \sum_{k=1}^T \frac{1}{k^{a}}$, when $a\in(0,1)$, one can easily show the following inequality hold
\begin{equation} \label{ieq: convex: rt}
    \begin{aligned}
        \sum_{k=1}^T \frac{1}{k^{a}} \leq 1 + \int_{1}^{T} k^{-a} \leq c_1T^{1-a},
    \end{aligned}
\end{equation}
where $c_1>0$ is some positive constant.

Similarly, consider $\sum_{k=1}^T  \frac{\alpha_k}{\mu_k^2}=\sum_{k=1}^T k^{2a-b} $, when $b-2a<1$, one can also prove that it satisfies the following inequality
\begin{equation} \label{ieq: convex: st}
    \begin{aligned}
        \sum_{k=1}^T k^{2a-b} \leq c_2 T^{2a-b+1},
    \end{aligned}
\end{equation}
where $c_2>0$ is some positive constant.

Next, let us define $f^T_{\text{best}}=\min_{k=\{0,1,\ldots,T\}}f(x_k)$. Consider the left hand side of the inequality \eqref{ieq: convex: after sum}, we have
\begin{equation} \label{ieq: convex: lhs}
    \sum_{k=1}^T2\Embb[f(x_k)-f(x_*)] \geq 2T\Embb[f^T_{\text{best}}-f(x_*)].
\end{equation}
Thus, combining \eqref{ieq: convex: after sum}, \eqref{ieq: convex: rt}, \eqref{ieq: convex: st}, and \eqref{ieq: convex: lhs}, one arrives at~\eqref{rate: convex}.

Finally, we will establish the optimal convergence bound achievable by the algorithm. Under the constraints~$b-2a<1$ and~$0<a<1$, we aim to maximize
\begin{equation*}
    t = \min_{a,b\in\Rmbb_+}\{a, 1 - b,\ -2a + b\},
\end{equation*}
One can easily show using an LP formulation that the maximum value of~$t$ is~$1/4$ with the optimal solution~$(a,b)=(1/4,3/4)$. This result shows that the best bound the SSG algorithm could achieve for the expected risk problem~\eqref{problem: f(x)} in the convex case is of the order of $1/T^{1/4}$.
\eproof

Before moving on to the non-convex case, we note that the best convergence rate of the SSG algorithm in the convex setting is of the order of~$T^{1/4}$, which matches the rate for a different non-smooth convex stochastic composition problem (see \cite{wang2017stochastic}), where the non-smoothness lies instead in the inner function~$\psi$. Faster rates in this related compositional setting typically rely on additional structure or acceleration (see~\cite{wang2017accelerating} and Appendix~C).
In Section~\ref{section: compositional} and~Section~\ref{section: finite-sum}, we will demonstrate how the non-smooth compositional problems~\eqref{prob:1} and~\eqref{prob:2} in the convex case can be fitted in the SSG algorithm as instances of the general problem \eqref{problem: f(x)}, thereby allowing the SSG method to achieve the developed convergence rate of the non-smooth convex compositional setting.

\begin{remark}
We also note that, in some special settings, Problem~\eqref{prob:1} can be
reformulated through a saddle-point representation. For example, when
$\phi(\cdot)=|\cdot|$, one has $|t|=\max_{y\in[-1,1]} yt$, and hence
Problem~\eqref{prob:1} can be written as
\begin{equation} \label{prob: example}
    \min_x \max_{y\in[-1,1]} \; \mathbb{E}[y\psi(x,\xi)].
\end{equation}
For such a formulation, standard stochastic primal--dual methods can achieve an
$\Ocal(T^{-1/2})$ rate in terms of the expected primal--dual gap of a weighted
average of the iterates, under a convex--concave assumption
(see~\cite{nemirovski2009robust,juditsky2011solving}). However, this rate is
not applicable in our setting because when the inner function
$\psi$ is nonlinear (for instance quadratic), the saddle-point reformulation is generally not
convex--concave, and thus the standard primal--dual convergence theory does not
directly apply. This situation commonly arises in machine learning
settings where $\psi$ represents a nonlinear prediction model.
\end{remark}

\subsection{Rate in the non-convex case} \label{subsection: non-convex} 
In this subsection, we examine the non-convex setting of the SSG method. The lack of convexity in this context introduces additional analytical challenges compared to the strongly convex and convex cases. We begin by outlining the assumptions necessary for the non-convex scenario and subsequently develop the corresponding convergence results for the SSG algorithm under these assumptions.

Assumption \ref{assumption: L-smoothness} requires that the gradient of the smoothing function is Lipschitz continuous with a constant inversely proportional to the smoothing parameter~$\mu$. This inverse relationship indicates that decreasing~$\mu$, thereby reducing the degree of smoothing and allowing the function to more closely approximate the original non-smooth objective, results in an increase in the Lipschitz constant. Most smoothing functions (see Section~\ref{section: smoothing function}) satisfy this property. We also require that the stepsize should be of a diminishing type, as shown in Assumption~\ref{assumption: non-convex: stepsize}.

\bassumption \label{assumption: L-smoothness}
For all~$x\in\Rmbb^n$ and for all~$\mu>0$, the gradient of the smoothing function~$\ftilde(x,\mu)$ is Lipschitz continuous in~$x$ with a Lipschitz constant~$C_0/\mu$ for some~$C_0>0$.
\eassumption
\bassumption \label{assumption: non-convex: stepsize}
The sequence of diminishing stepsizes~$\{\alpha_k\}$ satisfies~$\sum_{k=1}^\infty \alpha_k= \infty$, and $\sum_{k=1}^\infty \alpha_k^d<\infty$ for some~$d>1$.
\eassumption

Assumption~\ref{assumption: non-convex: scaling} requires the smoothing approximation to vary in a uniformly Lipschitz manner with respect to the smoothing parameter~$\mu$. This condition is used to control the change in the smoothing objective caused by moving from~$\mu_k$ to~$\mu_{k+1}$ across successive iterations. We will see in Sections~4 and~5 that this assumption is reasonable.
\bassumption \label{assumption: non-convex: scaling}
There exists a constant $C_1>0$ such that, for all $x\in\mathbb R^n$ and all $\mu,\mu'>0$,
\[
|\tilde f(x,\mu')-\tilde f(x,\mu)|
\;\le\; C_1 |\mu'-\mu|.
\]
\eassumption

After having listed all the required assumptions, we now establish a first convergence result for the non-convex setting on the gradient of the smoothing function. 
Notice that, in the non-convex case, unlike in the convex setting, we do not require the bounded feasible region assumption. Therefore, our analysis applies to a more general unconstrained setting, and no projection step is needed. Accordingly, we consider the unprojected update rule
$x_{k+1}=x_k-\alpha_k \nabla_x \tilde f(x_k,\xi_k,\mu_k)$
in Algorithm~1.
In what follows, we present the asymptotic convergence analysis, drawing inspiration from the proof in~\cite{TGiovannelli_GKent_LNVicente_2022}. 
\btheorem[Convergence in the non-convex case] \label{theorem: non-convex}
Under Assumptions~\ref{assumption: unbiasedness}--\ref{assumption: sc: diff} and~\ref{assumption: L-smoothness}--\ref{assumption: non-convex: scaling},  for~$d\in(1,2)$, by choosing~$\mu_k=\alpha_k^{-\frac{d}{3}+\frac{2}{3}}$, one has
\begin{equation} \label{lim: noncovex1}
    \lim_{T\rightarrow\infty}\Embb\left[\sum_{k=1}^T \alpha_k\|\nabla_x \ftilde(x_k,\mu_k)\|^2\right] \;<\; \infty.
\end{equation}
Moreover, with $A_T\coloneq \sum_{k=1}^T \alpha_k$, one has
\begin{equation} \label{lim: noncovex2}
    \lim_{T\rightarrow\infty}\Embb\left[\frac{1}{A_T}\sum_{k=1}^T \alpha_k\|\nabla_x \ftilde(x_k,\mu_k)\|^2\right] \;=\; 0.
\end{equation}
\etheorem
\bproof
From Assumptions~\ref{assumption: L-smoothness} and~\ref{assumption: non-convex: scaling}, and noting that the smoothing parameter decreases at each iteration in our setup, one has
\begin{equation*} \label{ieq: non-convex: Lsmooth}
\begin{aligned}
    \ftilde(x_{k+1},\mu_{k+1})-\ftilde(x_k,\mu_k) &= (\ftilde(x_{k+1},\mu_{k+1})-\ftilde(x_{k+1},\mu_k))+(\ftilde(x_{k+1},\mu_k)-\ftilde(x_k,\mu_k)) \\
    &\leq C_1(\mu_{k}-\mu_{k+1}) +\langle \nabla_x \ftilde(x_k,\mu_k),x_{k+1}-x_{k}\rangle + \frac{C_0}{2\mu_k}\|x_{k+1}-x_{k}\|^2.
\end{aligned}
\end{equation*}
Recalling the update rule of the SSG algorithm~$x_{k+1} = x_k-\alpha_k \nabla_x \ftilde(x_k,\xi_k,\mu_k)$, we have
\begin{equation} \label{ieq: non-convex: recall update}
\begin{aligned}
    \ftilde(x_{k+1},\mu_{k+1})-\ftilde(x_k,\mu_k) \;\leq\; &C_1(\mu_{k}-\mu_{k+1}) - \alpha_k  \langle\nabla_x\ftilde(x_k,\mu_k),\nabla_x \ftilde(x_k,\xi_k,\mu_k) \rangle \\
    &+ \frac{\alpha_k^2 C_0}{2\mu_k} \|\nabla_x \ftilde(x_k,\xi_k,\mu_k)\|^2.    
\end{aligned}
\end{equation}
Then, by taking conditional expectations with respect to the distribution of $\xi_k$ of \eqref{ieq: non-convex: recall update}, we obtain
\begin{equation} \label{ieq: non-convex: expectation_k}
\begin{aligned}
    \Embb[\ftilde(x_{k+1},\mu_{k+1})|\xi_k]-\ftilde(x_k,\mu_k)\;\leq\; &C_1(\mu_{k}-\mu_{k+1}) - \alpha_k \|\nabla_x\ftilde(x_k,\mu_k)\|^2 \\&+\frac{\alpha_k^2 C_0}{2\mu_k} \Embb[ \|\nabla_x \ftilde(x_k,\xi_k,\mu_k)\|^2|\xi_k].    
\end{aligned}
\end{equation}
Now we take the total expectation of \eqref{ieq: non-convex: expectation_k} and obtain
\begin{equation} \label{ieq: non-convex: before sum}
\begin{aligned}
    \Embb[\ftilde(x_{k+1},\mu_{k+1})] - \Embb[\ftilde(x_k,\mu_k)] \;\leq\; &C_1(\mu_{k}-\mu_{k+1})-\alpha_k \Embb[\|\nabla_x\ftilde(x_k,\mu_k)\|^2]\\&+\frac{\alpha_k^2 C_0}{2\mu_k} \Embb[ \|\nabla_x \ftilde(x_k,\xi_k,\mu_k)\|^2].    
\end{aligned}
\end{equation}

Next, by summing the inequality~\eqref{ieq: non-convex: before sum} from~$k=1$ to~$T$ for some~$T>0$, we have
\begin{equation} \label{ieq: non-convex: sum}
\begin{aligned}
    \Embb[\ftilde(x_{T+1},\mu_{T+1})] - \Embb[\ftilde(x_1,\mu_1)] \;\leq\; &C_1\sum_{k=1}^T(\mu_{k}-\mu_{k+1})  -\sum_{k=1}^T \alpha_k \Embb[\|\nabla_x\ftilde(x_k,\mu_k)\|^2] \\&+ \frac{C_0}{2} \sum_{k=1}^T \frac{\alpha_k^2}{\mu_k} \Embb[ \|\nabla_x \ftilde(x_k,\xi_k,\mu_k)\|^2].
\end{aligned}
\end{equation}
By defining $\ftilde^T_{\text{best}}=\min_{k=\{0,1,\ldots,T+1\}}\ftilde(x_k,
\mu_k)$ and applying Assumption~\ref{assumption: bounded variance}, we obtain
\begin{equation} \label{ieq: non-convex: after variance assumption}
\begin{aligned}
    \sum_{k=1}^T \alpha_k\Embb[\|\nabla_x\ftilde(x_k,\mu_k)\|^2] &\leq C_1(\mu_1-\mu_{T+1})+\Embb[\ftilde(x_1,\mu_1)] - \ftilde^T_{\text{best}} + \frac{C_0 G^2}{2}\sum_{k=1}^T \frac{\alpha_k^2}{\mu_k^3} \\
    &\leq C_1\mu_1 + \Embb[\ftilde(x_1,\mu_1)]+\frac{C_0 G^2}{2}\sum_{k=1}^T \frac{\alpha_k^2}{\mu_k^3}.
\end{aligned}
\end{equation}
By choosing $\mu_k = \alpha_k^{-\frac{d}{3} + \frac{2}{3}}$ for some $d\in(1,2)$, the diminishing stepsize assumption (Assumption~\ref{assumption: non-convex: stepsize}) ensures that the right-hand side of 
\eqref{ieq: non-convex: after variance assumption} converges to a finite limit 
as $T \to \infty$. Since the left-hand side is a nondecreasing sum (all terms are nonnegative), we prove that the theorem's first statement~\eqref{lim: noncovex1} holds.

Next, we turn to prove the second statement of the theorem. We divide both sides of \eqref{ieq: non-convex: after variance assumption} by $A_T=\sum_{k=1}^T \alpha_k$ and get
\begin{equation} \label{eq: nonconvex last}
\begin{aligned}
    \frac{\sum_{k=1}^T \alpha_k\Embb[\|\nabla_x\ftilde(x_k,\mu_k)\|^2]}{A_T} 
    \leq \frac{C_1(\mu_1-\mu_{T+1})}{A_T}+\frac{\Embb[\ftilde(x_1,\mu_1)] - \ftilde^T_{\text{best}}}{A_T} + \frac{\frac{C_0 G^2}{2}\sum_{k=1}^T \frac{\alpha_k^2}{\mu_k^3}}{A_T}.
\end{aligned}
\end{equation}
Taking the limit as $T\rightarrow\infty$, recalling~$\mu_k = \alpha_k^{-\frac{d}{3} + \frac{2}{3}}$, and using the assumption of the stepsize (Assumption~\ref{assumption: non-convex: stepsize}), we prove that the theorem's second statement~\eqref{lim: noncovex2} holds.
\eproof


While Theorem~\ref{theorem: non-convex} does not provide an explicit convergence rate, the following corollary specifies the rate when we assume the computational budget~$T$ is known in advance.
\bcorollary \label{coroll: nonconvex rate1}
Under Assumptions~\ref{assumption: unbiasedness}--\ref{assumption: sc: diff} and~\ref{assumption: L-smoothness}--\ref{assumption: non-convex: scaling}, for~$d\in(1,2)$ and~$b\in(1/d,1)$, by choosing~$\mu_k=\alpha_k^{-\frac{d}{3}+\frac{2}{3}}$ and~$\alpha_k=k^{-b}$, one further has
\begin{equation} \label{lim: noncovex3}
    \min_{1\leq k \leq T}\Embb\!\left[\|\nabla_x\ftilde(x_k,\mu_k)\|^2\right]=\Ocal(T^{-(1-b)}).
\end{equation}

Moreover, one can show that the best convergence rate of~\eqref{lim: noncovex3} can be arbitrarily close to the order of~$T^{-1/2}$.
\ecorollary
\bproof
Applying inequality~\eqref{eq: nonconvex last} with~$\alpha_k = k^{-b}$ for~$b \in (1/d,1)$, and noting that~$bd>1$, we obtain
\begin{equation} \label{eq: order}
A_T=\sum_{k=1}^T \alpha_k =\sum_{k=1}^Tk^{-b}=\Theta(T^{1-b})
\quad\text{and}\quad
\sum_{k=1}^T\frac{\alpha_k^2}{\mu_k^3}=\sum_{k=1}^T\alpha_k^d=\sum_{k=1}^Tk^{-bd} =\Ocal(1).
\end{equation}
Therefore, it follows from~\eqref{eq: nonconvex last} that
\[
\frac{\sum_{k=1}^T \alpha_k \Embb\!\left[\|\nabla_x\ftilde(x_k,\mu_k)\|^2\right]}{A_T}
=
\Ocal(T^{-(1-b)}).
\]
Since
\[
\min_{1\leq k\leq T}\Embb\!\left[\|\nabla_x\ftilde(x_k,\mu_k)\|^2\right]
\leq
\frac{\sum_{k=1}^T \alpha_k \Embb\!\left[\|\nabla_x\ftilde(x_k,\mu_k)\|^2\right]}{A_T},
\]
we conclude that
\[
\min_{1\leq k \leq T}\Embb\!\left[\|\nabla_x\ftilde(x_k,\mu_k)\|^2\right]
=
\Ocal(T^{-(1-b)}).
\]

We note that, for any fixed choice of~$d$, the best rate allowed by the assumptions is achieved by taking~$b$ arbitrarily close to the lower end of its admissible interval, which is of order~$1/d$. Consequently, the overall best rate is approached by taking~$d$ arbitrarily close to~$2$ and choosing~$b$ correspondingly close to~$1/2$. In this way, one obtains a rate arbitrarily close to~$T^{-1/2}$.
\eproof

Theorem~\ref{theorem: non-convex} and Corollary~\ref{coroll: nonconvex rate1} provide the convergence of the SSG method for gradients of the smoothing function~$\ftilde$. We want to further investigate the convergence for subgradients of the true non-smooth function~$f$. For this purpose, we impose Assumption~\ref{assumption: subgradient} which guarantees the existence of a nearby true subgradient within a distance of the order of~$\mu$ to any smoothing gradient.
\bassumption \label{assumption: subgradient}
For all~$x\in \Rmbb^n$, and for all~$\mu>0$, there exists a~$y\in\Bcal(x,E_1\mu)$ and a subgradient~$g(y)\in\partial f(y)$ such that 
\begin{equation} \label{ieq: non-convex: f-g}
    \|\nabla_x \ftilde(x,\mu)-g(y)\| \;\leq\; E_2\mu,
\end{equation}
where~$E_1$ and~$E_2$ are some positive constants.
\eassumption

Note that when~$f$ is locally Lipschitz continuous at a limit point~$x_*$, Assumption~\ref{assumption: subgradient} 
implies the gradient consistency property (see details in Section~\ref{section: smoothing function}). 
In fact, consider any sequence~$\{(x_k,\mu_k)\}$ satisfying~$x_k \to x_*~\text{and}~\mu_k \downarrow 0$, and let~$y_k$ be the corresponding points that satisfy Assumption~\ref{assumption: subgradient}. It is easy to show that~$y_k \to x_*$ as~$\mu_k \downarrow 0$.  Since~$f$ is locally Lipschitz continuous at~$x_*$, its Clarke subdifferential is outer semicontinuous at~$x_*$ (see~\cite[Theorem 5.19]{rockafellar1998variational}), which leads to~$g(y_k)\to g(x_*)\in\partial f(x_*)$. Thus, taking limits in~$\|\nabla_x \ftilde(x_k,\mu_k)-g(y_k)\|\leq E_2\mu_k$ yields
\begin{equation} \label{lim: gc}
    \lim_{x_k \to x_*,\,\mu \downarrow 0}\nabla_x \tilde{f}(x_k,\mu_k) = g(x_*)\in \partial f(x_*).
\end{equation}
Using~\eqref{lim: gc} and the fact that the Clarke subdifferential~$\partial f(x_*)$ is always a convex set, we have~$\operatorname{co} G_{\ftilde}(x_*)=\operatorname{co} \{\lim_{x \to x_*,\,\mu \downarrow 0}\nabla_x \tilde{f}(x,\mu)\} \subseteq \partial f(x_*)$. Since one always has~$\partial f(x_*)\subseteq\operatorname{co} G_{\ftilde}(x_*)$ for a locally Lipschitz continuous~$f$ (see in Section~\ref{section: smoothing function}), we have thus showed that Assumption~\ref{assumption: subgradient} implies the gradient consistency property. In addition, in Appendix~A, we directly explain how such a nearby point~$y_k$ can be found in the compositional case~$f(x)=\phi(\psi(x))$, thereby justifying the reasonableness of this assumption.

\bcorollary[Convergence in the non-convex case] \label{cor: non-convex}
Under Assumptions~\ref{assumption: unbiasedness}--\ref{assumption: sc: diff} and~\ref{assumption: L-smoothness}--\ref{assumption: subgradient}, for~$d\in(1,7/5]$, by choosing~$\mu_k=\alpha_k^{-\frac{d}{3}+\frac{2}{3}}$, one has
\begin{equation} \label{lim: cor: noncovex1}
    \lim_{T\rightarrow\infty}\Embb\left[\sum_{k=1}^T \alpha_k\|g(y_k)\|^2\right] \;<\; \infty.
\end{equation}
Moreover, with $A_T\coloneq \sum_{k=1}^T \alpha_k$, one has
\begin{equation} \label{lim: cor: noncovex2}
    \lim_{T\rightarrow\infty}\Embb\left[\frac{1}{A_T}\sum_{k=1}^T \alpha_k\|g(y_k)\|^2\right] \;=\; 0,
\end{equation}
where~$g(y_k)$ is a subgradient satisfying~\eqref{ieq: non-convex: f-g}.
\ecorollary
\bproof
For any~$y_k\in\Bcal(x_k,E_1\mu)$ such that the subgradient~$g(y_k)$ satisfies Assumption~\ref{assumption: subgradient}, we have
\begin{equation} \label{ieq: non-convex: cor: sub_square}
    \begin{aligned}
        \Embb\left[\sum_{k=1}^T \alpha_k\|g(y_k)\|^2\right] 
        &= \Embb\left[\sum_{k=1}^T \alpha_k\|g(y_k)-\nabla_x\ftilde(x_k,\mu_k)+\nabla_x\ftilde(x_k,\mu_k)\|^2\right] \\
        &\leq 2\Embb\left[\sum_{k=1}^T \alpha_k\|g(y_k))-\nabla_x\ftilde(x_k,\mu_k)\|^2\right] + 2\Embb\left[\sum_{k=1}^T \alpha_k\|\nabla_x\ftilde(x_k,\mu_k)\|^2\right] \\
        &\leq 2E_2^2\sum_{k=1}^T \alpha_k\mu_k^{2} + 2\Embb\left[\sum_{k=1}^T \alpha_k\|\nabla_x\ftilde(x_k,\mu_k)\|^2\right].
    \end{aligned} 
\end{equation}
Given our choice of $\mu_k=\alpha_k^{-\frac{d}{3}+\frac{2}{3}}$ with~$d\in(1,7/5]$, we observe that~$\lim_{T\rightarrow\infty}E_2^2\sum_{k=1}^T \alpha_k\mu_k^{2}<\infty$. Combining~$\lim_{T\rightarrow\infty}E_2^2\sum_{k=1}^T \alpha_k\mu_k^{2}<\infty$ with~\eqref{lim: noncovex1}, we know that the right-hand side of~\eqref{ieq: non-convex: cor: sub_square} is upper bounded when~$T\rightarrow\infty$. Since~$\sum_{k=1}^T \alpha_k\|g(y_k)\|^2$ is a nonnegative‐term sum, one can conclude that~$\eqref{lim: cor: noncovex1}$ holds.

Next, we prove the second statement of the theorem. Notice that~\eqref{lim: noncovex2} from Theorem~\ref{theorem: non-convex} also holds here. Consider~\eqref{ieq: non-convex: cor: sub_square} again and divide its both sides by~$A_T=\sum_{k=1}^T \alpha_k$, obtaining
\begin{equation} \label{ieq: non-convex: cor: sub/AT}
    \frac{\Embb\left[\sum_{k=1}^T \alpha_k\|g(y_k)\|^2\right]}{A_T} \leq \frac{2E_2^2\sum_{k=1}^T \alpha_k\mu_k^{2}}{A_T} + \frac{2\Embb\left[\sum_{k=1}^T \alpha_k\|\nabla_x\ftilde(x_k,\mu_k)\|^2\right]}{A_T}.
\end{equation}
When~$d \in (1,7/5]$, we use the finiteness of 
$\lim_{T \rightarrow \infty} E_2^2 \sum_{k=1}^T \alpha_k \mu_k^{2}$ 
together with~\eqref{lim: noncovex2} 
to establish the second statement~\eqref{lim: cor: noncovex2} of the theorem. 
\eproof

Similarly, by assuming the computational budget~$T$ is known in advance, we provide a convergence rate for~$g(y_k)$ extending the convergence result of Corollary~\ref{cor: non-convex} as follows.
\bcorollary \label{cor: non-convex-rate}
Under Assumptions~\ref{assumption: unbiasedness}--\ref{assumption: sc: diff} and~\ref{assumption: L-smoothness}--\ref{assumption: subgradient}, for~$d\in(1,7/5]$ and~$b\in(1/d,1)$, by choosing~$\mu_k=\alpha_k^{-\frac{d}{3}+\frac{2}{3}}$ and~$\alpha_k = k^{-b}$, one further has
\begin{equation} \label{eq: nonconvex-rate-min}
\min_{1\le k\le T}\Embb\left[\|g(y_k)\|^2\right]
=
\Ocal\left(T^{-r+\varepsilon}\right),
\end{equation}
where~$r=\min\left\{1-b,\; \frac{2b(2-d)}{3}\right\}$,~$\varepsilon$ is any positive number, and~$g(y_k)$ is a subgradient satisfying~\eqref{ieq: non-convex: f-g}.

Moreover, one can show that the best convergence rate of \eqref{eq: nonconvex-rate-min} can be arbitrarily close to the order of~$T^{-2/7+\varepsilon}$.
\ecorollary
\bproof
Recall that~$\ftilde^T_{\text{best}}=\min_{k=\{0,1,\ldots,T\}}\ftilde(x_k,\mu_k)$. From~\eqref{ieq: non-convex: after variance assumption} together with~\eqref{ieq: non-convex: cor: sub_square}, we obtain
\begin{equation} \label{ieq: nonconvex2 gk1}
\Embb\left[\sum_{k=1}^T \alpha_k\|g(y_k)\|^2\right]
\le
2E_2^2\sum_{k=1}^T \alpha_k\mu_k^2
+ 2C_1\mu_1
+ 2\Embb[\tilde f(x_1,\mu_1)]
+ C_0G^2\sum_{k=1}^T \frac{\alpha_k^2}{\mu_k^3}.
\end{equation}
Given the choice of~$\mu_k=\alpha_k^{\frac{2-d}{3}}$, the inequality~\eqref{ieq: nonconvex2 gk1} can be further written as
\begin{equation} \label{ieq: nonconvex2 gk2}
\Embb\left[\sum_{k=1}^T \alpha_k\|g(y_k)\|^2\right]
\le
2E_2^2\sum_{k=1}^T \alpha_k^{\frac{7-2d}{3}}
+ 2C_1\alpha_1^{\frac{2-d}{3}}
+ 2\Embb[\tilde f(x_1,\mu_1)]
+ C_0G^2\sum_{k=1}^T \alpha_k^d.
\end{equation}

Since $\alpha_k=k^{-b}$ with $b\in(1/d,1)$, recall from Corollary~\ref{coroll: nonconvex rate1} that we have~\eqref{eq: order}. Moreover, one can derive that
\begin{equation} \label{nonconvex fact2}
\sum_{k=1}^T \alpha_k^{\frac{7-2d}{3}}
=
\sum_{k=1}^T k^{-b\frac{7-2d}{3}}
=
\Ocal\left(T^{\max\left\{1-b\frac{7-2d}{3},\,0\right\}}\log T\right).
\end{equation}
Now dividing both sides of~\eqref{ieq: nonconvex2 gk2} by $A_T$, and combining with~\eqref{eq: order} and~\eqref{nonconvex fact2}, we obtain that
\[
\frac{1}{A_T}\Embb\left[\sum_{k=1}^T \alpha_k\|g(y_k)\|^2\right]
=\Ocal\left(T^{-\min\left\{1-b,\;\frac{2b(2-d)}{3}\right\}}\log T\right) + 
\Ocal\left(T^{-(1-b)}\right)
.
\]
By denoting~$r=\min\left\{1-b,\;\frac{2b(2-d)}{3}\right\}$, and for any~$\varepsilon>0$, we obtain
\[
\frac{1}{A_T}\Embb\left[\sum_{k=1}^T \alpha_k\|g(y_k)\|^2\right]
= \Ocal(T^{-r+\varepsilon}).
\]

Finally, since
\[
\min_{1\le k\le T}\Embb\left[\|g(y_k)\|^2\right]
\le
\frac{1}{A_T}\sum_{k=1}^T \alpha_k \Embb\left[\|g(y_k)\|^2\right],
\]
we conclude that
\[
\min_{1\le k\le T}\Embb\left[\|g(y_k)\|^2\right]
=
\Ocal\left(T^{-r+\varepsilon}\right).
\]
 Lastly, one can verify that the exponent~$r$ is bounded above by~$2/7$. In fact, by taking~$d=\frac{7}{5}$ and letting~$b\to \frac{5}{7}^{+}$, one obtains the best achievable rate, in the sense that it is of order~$\Ocal(T^{-2/7+\varepsilon})$ for any arbitrarily small~$\varepsilon>0$.
\eproof

Corollaries~3.1 and~3.3 make the non-convex rate explicit in our setting. We emphasize, however, that these rates are stated in terms of the smoothing gradient and of nearby true subgradients~$g(y_k)$ provided by Assumption~3.9.

We further note that our best rate~$\Ocal(T^{-2/7+\varepsilon})$ can match arbitrarily closely the non-convex rate obtained in~\cite{wang2017stochastic}, at the expense of any small~$\varepsilon>0$. That work, again, considers a related compositional problem in which the non-smoothness arises from the inner function, whereas in our setting it arises from the outer function, and it is based on a different algorithmic framework. From a structural viewpoint, the setting in this study is complementary to ours. In~~\cite{wang2017stochastic}, the main difficulty comes from the nested expectation structure together with the possible non-smoothness of the inner mapping, whereas in our paper the inner function remains differentiable and the non-smoothness appears in the outer function. Accordingly, our method smooths the outer function while preserving the compositional structure, whereas SCGD-type methods rely on auxiliary tracking sequences to approximate the inner expectation. Appendix~C further strengthens this connection by adopting a two-timescale idea inspired by this study for the strongly convex case~$f(x)=\phi(\mathbb{E}[\psi(x,\xi)])$.

\section{Case~$f(x)=\phi(\Embb[\psi(x,\xi)])$} \label{section: compositional}
An objective function of the type~\eqref{problem: f(x)} takes a compositional form in many real-world applications, such as risk-averse optimization. An outer function $\phi(\cdot)$ is composed with an inner function~$\Embb[\psi(\cdot,\xi)]$, leading to the problem formulation~\eqref{prob:1}. In this case, the smoothing function of~$f$ (with~$\psi(x)=\Embb[\psi(x,\xi)]$) is 
$$\ftilde(x,\mu)\;=\; \tilde{\phi}(\psi(x),\mu).$$ And the gradient of~$\nabla_x \ftilde(x,\mu)$ is
\begin{equation*}
   \nabla_x \ftilde(x,\mu) \;=\; \tilde{\phi}'(\psi(x),\mu)\nabla\psi(x). 
\end{equation*}
To compute the smoothing gradient as an unbiased estimator of~$\nabla_x \ftilde(x,\mu)$ (to satisfy Assumption~\ref{assumption: unbiasedness}), we  independently sample two batches~$\xi^1$ and~$\xi^2$ from the random vector~$\xi$ and then compute~$\psi(x_k,\xi^1_k)$ and~$\nabla_x \psi(x_k,\xi^2_k)$ separately. By doing so, the smoothing gradient estimator is written as follows
\begin{equation} \label{eq: grad_ftilde}
    \nabla_x \ftilde(x,\xi^1,\xi^2,\mu) \;=\; \tilde{\phi}'(\psi(x,\xi^1),\mu)\nabla_x \psi(x,\xi^2),
\end{equation}
where~$\tilde{\phi}'(\cdot,\mu)$ denotes the derivative of the smoothing function~$\tilde{\phi}$ evaluated at~$\psi(x,\xi^1)$.
We will show in Proposition~\ref{proposition: compositional: unbias} that~$\nabla_x \ftilde(x,\xi^1,\xi^2,\mu)$ is indeed an unbiased estimator of~$\nabla_x \ftilde(x,\mu)$.

To deal with this compositional setting, we present a particular SSG method in Algorithm~\ref{algorithm: SSG for general compositional}. We assume that the sequences~$\{\alpha_k\}$ and~$\{\mu_k\}$ are both decreasing and converging to zero.
An initial point~$x_1$, along with a sequence of stepsize~$\{\alpha_k\}$, and a sequence of smoothing parameter~$\{\mu_k\}$ are required as input. At each iteration, the algorithm samples realizations~$\xi^1_k$ and~$\xi^2_k$ of the random variable~$\xi$ independently and then computes $\psi(x_k,\xi^1_k)$ and~$\nabla_x \psi(x_k,\xi^2_k)$ separately.
Next, the algorithm computes smoothing gradient~$\nabla_x \ftilde(x,\xi^1,\xi^2,\mu)$ as in~\eqref{eq: grad_ftilde}. This smoothing gradient~$\nabla_x \ftilde(x,\xi^1,\xi^2,\mu)$ is used in a standard stochastic gradient descent update procedure, scaled by the stepsize~$\alpha_k$. Lastly, the SSG algorithm updates the smoothing parameter~$\mu$ at the end of each iteration.

\begin{algorithm}[H]
\caption{SSG method for case~$f(x)=\phi(\Embb[\psi(x,\xi)]
)$} \label{algorithm: SSG for general compositional}
\label{algorithm: ssg}
\begin{algorithmic}[1]
\medskip
\item[] \textbf{Input:} $x_1$,~$\{\alpha_k\}$, and~$\{\mu_k\}$.
\medskip
\item[] \textbf{For} $k = 1,2,\dots$ \textbf{do}
\begin{enumerate}
    \item[]  {\bf Step 1.}
    Generate realizations $\xi^1_k$ and $\xi^2_k$ of the random variable $\xi$. Then compute $
        \psi(x_k,\xi^1_k)$ and~$\nabla_x \psi(x_k,\xi^2_k)$, respectively.
    
    \item[]  {\bf Step 2.} 
    Compute~$\nabla_x \ftilde(x_k,\xi_k^1,\xi_k^2,\mu_k)= \tilde{\phi}'(\psi(x_k,\xi^1_k), \mu_k)\,\nabla_x \psi(x_k,\xi^2_k)$.
    \item[]  {\bf Step 3.} 
    Update iterate~$x_{k+1} = x_k -\alpha_k\nabla_x \ftilde(x_k,\xi_k^1,\xi_k^2,\mu_k)$.
    \item[]  {\bf Step 4.} 
    Update smoothing parameter~$\mu_{k+1} \leq \mu_k$.
\end{enumerate}
\item[] \textbf{End do}
\end{algorithmic}
\end{algorithm}

In this section, we will investigate under what properties of $\phi$ or $\tilde{\phi}$ and $\psi$ the assumptions needed in Section~\ref{section: general} are met. 
\subsection{About convexity} 



In many real-world optimization problems under uncertainty, the non-smoothness of the outer function~$\phi$ arises from absolute values. For example, problems that we are looking at could be of the form~$|\Embb[\psi(x,\xi)]|$. Consider a machine learning fairness scenario as an instance, where the inner function~$\psi$ measures the covariance between a prediction and a sensitive attribute under a random data sample~$\xi$, and the objective~$|\Embb[\psi(x,\xi)]|$ penalizes the magnitude of the expected fairness violation.

To smooth the outer function~$\phi$ with an absolute-value type function~$|\cdot|$, we can use smoothing function provided in Section~\ref{section: smoothing function}. In such a case,~$\tilde{\phi}(t,\mu)$ is convex with respect to~$t$. Then, if one has that~$\psi(x)=\Embb[\psi(x,\xi)]$ is affine in~$x$, as often occurs when~$\psi(x)$ represents a linear predictor in machine learning settings, the convexity assumption of the compositional smoothing function~$\ftilde(x,\mu)=\tilde{\phi}(\psi(x),\mu)$ is satisfied.

\subsection{About stochasticity}
Next, we will show that the bounded second moment assumption given in Section~\ref{section: general} (Assumption~\ref{assumption: bounded variance}) can be met in this compositional case. For this purpose, we first require that the second moment of the gradient of the inner function be bounded in expectation. Such an assumption is typically met in machine learning models (for example, linear regression and logistic regression)  with bounded feasible regions~$\Xcal\subset\Rmbb^n$.
\bassumption \label{assumption: compositional: nabla_psi} 
For all~$x\in\Xcal$, the gradient of the inner function~$\psi(x,\xi)$ has a bounded second moment
    \begin{equation*}
        \Embb[\|\nabla_x \psi(x,\xi)\|^2] \;\leq\; G_1^2,
    \end{equation*}
for some~$G_1>0$.
\eassumption

We also need to assume that the expected squared derivative of $\tilde{\phi}$ does not exceed a positive quantity that inversely depends on the smoothing parameter~$\mu$.

\begin{assumption} \label{proposition: compositional: tilde_phi}
For all~$x\in\Xcal$ and for all~$\mu>0$, the gradient of~$\Tilde{\phi}(\psi(x,\xi),\mu)$ satisfies
    \begin{equation*}
        \Embb[\Tilde{\phi}'(\psi(x,\xi),\mu)^2] \;\leq\; \left(\frac{G_2}{\mu}\right)^2,
    \end{equation*}
for some~$G_2>0$.
\end{assumption}

We note that Assumption~\ref{proposition: compositional: tilde_phi} is satisfied when using smoothing functions $\tilde{\phi}$ of the types shown in Section~\ref{section: smoothing function} (see~\eqref{function: |.|} and~\eqref{function: Hinge}), which correspond to functions of type~$\lvert\cdot\rvert$ and~$\max\{0,\cdot\}$. In such cases,~$\phi'(\psi(x,\xi),\mu)$ is a multiple of~$1/\mu$ of~$\psi(x,\xi)$, and Assumption~\ref{proposition: compositional: tilde_phi} is met as long as~$\psi(x,\xi)$ is bounded.
Hence, Assumption~\ref{proposition: compositional: tilde_phi} is a combination of the boundness of $\psi(x,\xi)$ and the properties of the smoothing function~$\tilde{\phi}(\cdot,\mu)$.

Using these two assumptions, we can ensure the bounded second moment of the smoothing gradient and thus confirm that Assumption~\ref{assumption: bounded variance} is valid in this compositional setting.
\begin{proposition} Suppose Assumptions \ref{assumption: compositional: nabla_psi} and~\ref{proposition: compositional: tilde_phi} hold. Additionally, assume that~$\psi(x,\xi^1)$ is sampled independently from $\nabla_x \psi(x,\xi^2)$. Then, the smoothing gradient satisfies 
\begin{equation*}         \Embb[\|\Tilde{\phi}'(\psi(x,\xi^1),\mu)\nabla_x\psi(x,\xi^2)\|^2] \;\leq\; \left(\frac{G_1G_2}{\mu}\right)^2.  
\end{equation*} \end{proposition}
\begin{proof}
    By the definition of the smoothing gradient estimate~\eqref{eq: grad_ftilde}, and since~$\tilde{\phi}'(\psi(x,\xi^1),\mu)$ is a scalar, together with the independence of~$\xi^1$ and~$\xi^2$, we have the following equalities
    \begin{equation*}
        \begin{aligned}
\mathbb{E}[\|\tilde{\phi}'(\psi(x,\xi^1),\mu) \nabla_x \psi(x,\xi^2)\|^2] &= \mathbb{E}[|\tilde{\phi}'(\psi(x,\xi^1),\mu)|^2 \|\nabla_x \psi(x,\xi^2)\|^2] \\
&= \mathbb{E}_{\xi^1}[|\tilde{\phi}'(\psi(x,\xi^1),\mu)|^2] \mathbb{E}_{\xi^2}[\|\nabla_x \psi(x,\xi^2)\|^2].
        \end{aligned}
    \end{equation*}
The final bound can be determined from Assumptions~\ref{assumption: compositional: nabla_psi} and \ref{proposition: compositional: tilde_phi}.
\end{proof}

Having established the bounded second moment of the smoothing gradient, we now proceed to demonstrate the unbiasedness condition needed for the smoothing gradient estimator, where the linearity of~$\Tilde{\phi}'(\cdot,\mu)$ is required in the proof of Proposition~4.2 (to confirm that Assumption~\ref{assumption: unbiasedness} can also be satisfied in this setting). Notice that since the smoothing technique considered for the non-smooth kink point of the absolute-value function yields a quadratic smoothing function, its derivative~$\Tilde{\phi}'(\cdot,\mu)$ is linear in the smoothing region (see~\eqref{function: |.|} for the $L_1$ loss). Thus, in our setting, for Step~3 in Algorithm~2, if the sampling of~$\xi_k^1$ at each iteration~$k$ is restricted so that~$\psi(x_k,\xi_k^1)$ falls within the smoothing region of~$\Tilde{\phi}'(\cdot,\mu_k)$ almost surely, then~$\Tilde{\phi}'(\psi(x_k,\xi_k^1),\mu_k)$ is evaluated in its linear region. Moreover, to ensure that this restriction does not introduce bias in the estimation of~$\psi(x_k)$, we require the restricted sampling rule to preserve the mean, namely~$\Embb_{\mathrm{res}}[\psi(x_k,\xi_k^1)] = \psi(x_k)$, where~$\Embb_{\mathrm{res}}$ denotes expectation under the restricted sampling rule. Under this condition, the restricted sample~$\psi(x_k,\xi_k^1)$ remains an unbiased estimator of~$\psi(x_k)$. We note, however, that this argument does not extend directly to the hinge case, since the corresponding smoothing region is one-sided and the same mean-preservation property is not naturally guaranteed by symmetry. 

\begin{proposition} \label{proposition: compositional: unbias}
     Suppose Assumptions \ref{assumption: compositional: nabla_psi} and~\ref{proposition: compositional: tilde_phi} hold. Let~$\Tilde{\phi}'(t,\mu)$ be a linear function in~$t$. Assuming the unbiasedness of~$\psi(x,\xi^1)$ and~$\nabla\psi(x,\xi^2)$, the smoothing gradient~$\nabla_x \ftilde(x,\xi^1,\xi^2,\mu)=\Tilde{\phi}'(\psi(x,\xi^1),\mu)\nabla_x\psi(x,\xi^2)$ serves as an unbiased stochastic estimator of $\nabla_x \ftilde(x,\mu)$, meaning that
    \begin{equation*} 
        \nabla_x \ftilde(x,\mu) \;=\; \Embb [\Tilde{\phi}'(\psi(x,\xi^1),\mu)\nabla_x\psi(x,\xi^2)].
    \end{equation*}   
\end{proposition}
\begin{proof} 
Again, since we sample~$\xi^1$ and~$\xi^2$ independently, we have
\begin{equation} \label{eq: prop unbias}
    \Embb [\Tilde{\phi}'(\psi(x,\xi^1),\mu)\nabla_x\psi(x,\xi^2)] = \Embb_{\xi^1} [\Tilde{\phi}'(\psi(x,\xi^1),\mu)]\Embb_{\xi^2}[\nabla_x\psi(x,\xi^2)]
\end{equation}
    Then using the lineality of~$\Tilde{\phi}'(\cdot,\mu)$ and the unbiasedness of~$\psi(x,\xi^1)$ and~$\nabla\psi(x,\xi^2)$ on \eqref{eq: prop unbias} completes the proof.
\end{proof}

\subsection{About smoothing}
Lastly, to fit the convergence theory that we discussed in Section~\ref{subsection: convex} into this compositional case, we also need to make sure that~$\ftilde$ is a smoothing function with an accuracy of the order of~$\mu$. This is guaranteed in Proposition~\ref{proposition: compositional: diff} as long as the smoothing function~$\tilde{\phi}$ satisfies the same property (Assumption~\ref{assumption: compositional: diff}).
\bassumption \label{assumption: compositional: diff}
For all $y\in\Rmbb$ and for all $\mu>0$, the difference between $\phi(y)$ and $\Tilde{\phi}(y,\mu)$ is bounded by some constant times the smoothing parameter in the following way
\begin{equation*}
    |\phi(y)-\Tilde{\phi}(y,\mu)|\;\leq\; C\mu,
\end{equation*}
where~$C>0$ is some positive constant.
\eassumption
\begin{proposition} \label{proposition: compositional: diff}
       Under Assumption \ref{assumption: compositional: diff}, for all~$\mu>0$, the smoothing function~$\ftilde(x,\mu)$ satisfies Assumption~\ref{assumption: sc: diff}.
\end{proposition}
\begin{proof} 
By Assumption~\ref{proposition: compositional: diff}, the smoothing accuracy assumption of~$\ftilde(x,\mu)$ is naturally satisfied as follows
\begin{equation*}
    |f(x)-\ftilde(x,\mu)| = |\phi(\psi(x))-\tilde{\phi}(\psi(x),\mu)|\leq C\mu.
\end{equation*}
 \end{proof}

Notice that as a consequence of Proposition~\ref{proposition: compositional: diff}, one can infer that $\ftilde(x,\mu)$ is indeed a smoothing function for $f(x)$ (see Definition~\ref{def: smoothing funtion}).

Building upon established assumptions and propositions of this section, we have demonstrated that the SSG method and its convergence results hold in a convex compositional setting of type $f(x)=\phi(\Embb[\psi(x,\xi)])$. Specifically, Proposition~\ref{assumption: compositional: nabla_psi} and~\ref{proposition: compositional: tilde_phi} ensure the bounded second moment and unbiasedness of the smoothing gradient estimates, and Proposition~\ref{assumption: compositional: diff} ensures the accuracy of the smoothing function. Consequently, our analysis paves the way for practical application in diverse optimization problems characterized by such a convex compositional structure. 

\section{Finite-sum compositional case} \label{section: finite-sum}
In this section, we will focus on case~$f(x)=\Embb[\phi(\psi(x,\xi))]$. In particular, we are interested in a finite-sum compositional structure because of its wide applicability. In many real-world applications, particularly in machine learning and large-scale optimization, the objective function naturally takes a finite-sum compositional structure. The finite-sum compositional objective function is defined in the following way
\begin{equation} \label{eq: finite-sum: f(x)}
    f(x) \;=\; \frac{1}{N} \sum_{i=1}^N \phi(\psi(x,\xi_i)),
\end{equation}
where $N \in \mathbb{N}_+$ denotes the number of components and $\xi_i$ represents the $i$-th data sample. Note that this finite-sum compositional function~\eqref{eq: finite-sum: f(x)} approximate function~$f(x)=\Embb[\phi(\psi(x,\xi))]$ using finite samples.

We then define the smoothing function~$\ftilde$ of the finite-sum compositional function~$f$ as follows
\begin{equation*} 
    \tilde{f}(x,\mu) \;=\; \frac{1}{N} \sum_{i=1}^N \tilde{\phi}(\psi(x,\xi_i),\mu),
\end{equation*}
where $\tilde{\phi}(\psi(x,\xi_i),\mu)$ serves as a smoothing function of $\phi(\psi(x,\xi_i))$ for each $i \in \{1,2,\ldots,N\}$. To compute the stochastic gradient in this finite-sum setting, we use a stochastic batch gradient approach
\begin{equation} \label{eq: finite-sum: gradient compute}
    \tilde{g}(x_k,\mu_k) \;=\; \frac{1}{|B_k|}\sum_{i \in B_k} \tilde{\phi}'(\psi(x_k,\xi_i),\mu_k)\nabla_x \psi(x_k,\xi_i),
\end{equation}
where $B_k$ denotes the sample batches. As a result, the update rule of the SSG method for the finite-sum compositional scenario is formulated as~$x_{k+1} = x_k - \alpha_k \tilde{g}(x_k,\mu_k)$.

After pointing out the modifications of the SSG method for the finite-sum compositional setting, we now present the corresponding particular method in Algorithm~$\ref{algorithm: SSG for finite sum compositional}$.

\begin{algorithm}[H]
    \caption{SSG method for the finite-sum compositional case} \label{algorithm: SSG for finite sum compositional}
\begin{algorithmic}[1]
\medskip
\item[] \textbf{Input:} $x_1$,~$\{\alpha_k\}$, and~$\{\mu_k\}$.
\medskip
\item[] \textbf{For} $k = 1,2,\dots$ \textbf{do}
\begin{enumerate}
    \item[]  {\bf Step 1.}
    Select a batch~$B_k$ from~$\{1,2,\dots,N\}$. For all $i\in B_k$, compute~$\psi(x_k,\xi_i)$ and~$\nabla_x \psi(x_k,\xi_i)$.
    
    \item[]  {\bf Step 2.} 
    Use~\eqref{eq: finite-sum: gradient compute} to compute the stochastic batch smoothing gradient~$\tilde{g}(x_k,\mu_k)$.
    \item[]  {\bf Step 3.}
    Update iterate~$x_{k+1} = x_k -\alpha_k\gtilde(x_k,\mu_k)$.
    \item[]  {\bf Step 4.}
    Update smoothing parameter~$\mu_{k+1} \leq \mu_k$.
\end{enumerate}
\item[] \textbf{End do}
\end{algorithmic}
\end{algorithm}

\subsection{About convexity} ~\label{subsec: aboutconvexity}
In the context of machine learning, the outer function
$\phi(\cdot)$ often represents a loss function that measures the discrepancies between
predictions and observed data. Many commonly used loss functions, such as $L$1 loss, hinge loss, and quantile loss, are convex but non-smooth. In these cases, by selecting an appropriate smoothing function $\tilde{\phi}(\cdot,\mu)$, it is possible to reformulate the compositional problem to fit the convex framework discussed in Section~\ref{section: general}.

For example, consider the setting where the outer function is the hinge loss. Let $\xi = (u, v)$ represent a pair of features and labels. The inner function is a linear function defined as $\psi(x,\xi) = x^\T u + b - v$. Suppose we have $N$ realizations of~$\xi = (u, v)$ denoted by $\xi_i=(u_i,v_i)$, where~$i\in\{1,\ldots,N\}$. In such a setting, the problem can be expressed as 
\begin{equation*} \label{eq: compostional: example1}
    f(x) \;=\; \frac{1}{N} \sum_{i=1}^N \phi(\psi(x,\xi_i)) \;=\; \frac{1}{N} \sum_{i=1}^N \max\{0,x^\T u_i + b - v_i\}.
\end{equation*} 
A popular smoothing function for~$\max\{0,\cdot\}$ was described in Section~\ref{section: smoothing function} (see \eqref{function: Hinge}). In such a case,~$\tilde{\phi}(t,\mu)$ is nondecreasing with respect to nonnegative~$t$. 
If the inner function~$\psi$ is convex, then for any realization~$\xi_i$ of $\xi$,~$\tilde{\phi}(\psi(x,\xi_i),\mu)$ is convex. Consequently, the sum of convex functions in~$\tilde{f}(x,\mu) = \frac{1}{N} \sum_{i=1}^N \tilde{\phi}(\psi(x,\xi_i),\mu)$ is convex and Assumption~\ref{assumption: convex} is guaranteed.

Another example is the~$L$1 loss, where we have 
\begin{equation*}
    f(x) \;=\; \frac{1}{N} \sum_{i=1}^N \phi(\psi(x,\xi_i)) \;=\; \frac{1}{N} \sum_{i=1}^N |x^\T u_i + b - v_i|.
\end{equation*}
By using the smoothing function described in Section~\ref{section: smoothing function} (see \eqref{function: |.|}), if~$\psi(x,\xi_i)$ is always evaluated with nonnegative values, we can also obtain the convexity of $\tilde{\phi}(\psi(x,\xi_i),\mu)$, therefore also meet the convexity of~$\ftilde(x,\mu)$ required in Assumption~\ref{assumption: convex}.

\subsection{About stochasticity}
In this subsection, we introduce the specific assumptions and propositions relevant to the finite-sum compositional case, thereby showing that this case can be fitted into the convergence result established in Section~\ref{subsection: convex}. Notice that the sampling rule used here differs from the one in Section~4 because the two problem classes have different structures. In Section~4, independent samples are needed to preserve unbiasedness, whereas in the finite-sum case considered here, uniform mini-batch sampling is sufficient to obtain an unbiased stochastic estimator. We now state this requirement formally in Assumption~\ref{assumption: finite-sum: batch}.
\bassumption \label{assumption: finite-sum: batch}
The mini-batch $B$ is sampled uniformly at random from $\{1,2,\ldots,N\}$, with each index $i$ equally likely to be selected without replacement.
\eassumption

This standard uniform sampling mechanism is crucial for maintaining the unbiasedness of the stochastic gradient estimates derived from the mini-batch. Thus, under Assumption \ref{assumption: finite-sum: batch}, we have the following proposition.
\begin{proposition} \label{proposition: finite-sum: unbiasedness}
    Under Assumption \ref{assumption: finite-sum: batch}, the smoothing gradient $\gtilde(x,\mu)$ given in~\eqref{eq: finite-sum: gradient compute} serves as an unbiased stochastic estimator of~$\nabla_x \ftilde(x,\mu)$, satisfying
    \begin{equation*}
        \nabla_x \ftilde(x,\mu) \;=\; \Embb[\gtilde(x,\mu)].
    \end{equation*}
\end{proposition}
\begin{proof} By the chain rule, we can express $\nabla_x \ftilde(x,\mu)$ as \begin{equation*} \nabla_x \tilde{f}(x,\mu) = \frac{1}{N} \sum_{i=1}^N \tilde{\phi}'(\psi(x,\xi_i), \mu) \nabla_x \psi(x,\xi_i).
\end{equation*} The stochastic batch gradient estimator is given by~\eqref{eq: finite-sum: gradient compute}. By taking the expectation over the random sampling of 
$B_k$ and utilizing Assumption \ref{assumption: finite-sum: batch}, which guarantees a uniform selection probability for each index $i$, the unbiasedness directly follows from standard results in stochastic gradient methods (see, e.g., \cite{bottou2018optimization}).
\end{proof}

Assumption \ref{assumption: finite-sum: nabla_psi} ensures that the gradient of every inner function $\psi(x,\xi_i)$ has a bounded second moment.
\bassumption \label{assumption: finite-sum: nabla_psi}
For all $x\in\Xcal$, the gradient of $\psi(x,\xi_i)$ is bounded for each $i\in\{1,2,\ldots,N\}$ by
    \begin{equation*}
        \Embb[\|\nabla_x \psi(x,\xi_i)\|^2] \;\leq\; G_1^2,
    \end{equation*}
for some $G_1>0$.
\eassumption

Similar to Section~\ref{section: compositional}, Assumption \ref{assumption: finite-sum: tilde_phi'} guarantees that, given the boundedness of the decision variables, the expected squared derivative of the smoothing function $\tilde{\phi}$ is controlled by a number that depends inversely on the smoothing parameter $\mu$.
\bassumption \label{assumption: finite-sum: tilde_phi'}
For all $\mu>0$, the gradient of $\Tilde{\phi}(\psi(x,\xi_i),\mu)$ satisfies
    \begin{equation*}
        E[\Tilde{\phi}'(\psi(x,\xi_i),\mu)^2] \;\leq\; \left(\frac{G_2}{\mu}\right)^2,
    \end{equation*}
for some $G_2>0$.
\eassumption

Proposition \ref{proposition: finite-sum: bounded graident} states that the second moment of the smoothing stochastic batch gradient is bounded under the assumptions listed above.
\begin{proposition} \label{proposition: finite-sum: bounded graident}
    Suppose Assumption \ref{assumption: finite-sum: batch}--\ref{assumption: finite-sum: tilde_phi'} hold. Additionally, assume that~$\psi(x,\xi_i)$ is sampled independently from $\nabla_x \psi(x,\xi_i)$ for each $i$. Then, the smoothing gradient estimate satisfies the following condition
    \begin{equation} \label{prop: finite-sum: bound}
    \begin{aligned}
        E[\|\gtilde(x,\mu)\|^2] 
        &\;\leq\;  \left(\frac{G_1G_2}{\mu}\right)^2.
    \end{aligned}
    \end{equation}
\end{proposition}
\begin{proof}
    The stochastic gradient estimate is given by
    \[
        \gtilde(x,\mu) = \frac{1}{|B|} \sum_{i \in B} g_i, \quad \text{where} \quad g_i = \widetilde{\phi}'(\psi(x,\xi_i),\mu) \nabla_x \psi(x,\xi_i),
    \]  
    from which we have
    \begin{equation} \label{eq: finite-sum: gtilde}
               \mathbb{E}\left[\|\gtilde(x,\mu)\|^2\right] = \mathbb{E} \left[ \left\|\frac{1}{|B|} \sum_{i \in B} g_i \right\|^2 \right] \leq \frac{1}{|B|^2} |B| \sum_{i \in B} \mathbb{E} \left[ \|g_i\|^2 \right]. 
    \end{equation}
    From Assumption~\ref{assumption: finite-sum: nabla_psi} and~\ref{assumption: finite-sum: tilde_phi'}, each $ g_i $ satisfies
    \[
        \mathbb{E} \left[ \|g_i\|^2 \right] \leq \mathbb{E} \left[\left(\widetilde{\phi}'(\psi(x,\xi_i),\mu)\right)^2 \|\nabla_x \psi(x,\xi_i)\|^2\right] \leq \left(\frac{G_1 G_2}{\mu}\right)^2.
    \]
    Substituting the above bound into \eqref{eq: finite-sum: gtilde}, we conclude \eqref{prop: finite-sum: bound}.
\end{proof}

\subsection{About smoothing}
Lastly, to fit the convergence theory that we discussed in the previous section into this case, we need to consider Assumption~\ref{assumption: compositional: diff} on the bound of the difference between each function $\phi(\cdot)$ and each smoothing function $\Tilde{\phi}(\cdot,\mu)$. The final result concerns the approximation accuracy of the smoothing function.
\begin{proposition}
       Under Assumption \ref{assumption: compositional: diff}, we have
    \begin{equation*}
    \begin{aligned}
        |f(x)-\ftilde(x,\mu)| \leq C\mu.
    \end{aligned}
    \end{equation*}
\end{proposition}
\begin{proof}
    Using Assumption \ref{assumption: compositional: diff}, it is easy to show that
    \begin{equation*}
    \begin{aligned}
        |f(x)-\ftilde(x,\mu)| &= \left|\frac{1}{N}\sum_{i=1}^N \left(\phi(\psi(x,\xi_i))-\Tilde{\phi}(\psi(x,\xi_i),\mu)\right)\right| \\
        &\leq \frac{1}{N} \sum_{i=1}^N \left|  \phi(\psi(x,\xi_i))-\Tilde{\phi}(\psi(x,\xi_i),\mu)  \right| \\
        &\leq C\mu.        
    \end{aligned}
    \end{equation*}
\end{proof}

Using the assumptions and propositions outlined in this section, we can fit the finite-sum compositional case into our convergence results presented in Section~\ref{subsection: convex}.


\section{Numerical results} \label{section: numerical}
In this section, we are going to report some preliminary numerical results on the performance of the SSG method for the finite-sum compositional case using Algorithm~\ref{algorithm: SSG for finite sum compositional}. 

\subsection{Results on hinge loss}
We will use a hinge loss as the outer function and a linear model as the inner function (see Section~\ref{subsec: aboutconvexity} for the detailed problem formulation). The SSG is run using the smoothing version of hinge given in~\eqref{function: Hinge}, and it is compared against a standard smooth SGD method ignoring the non-differentiability of the hinge-loss at the non-smooth point.

We also compare our SSG method with the stochastic compositional
gradient descent (SCGD) method~\cite{wang2017stochastic}. We note again that the
two methods are originally developed under different assumptions: our setting
allows a non-smooth outer function, such as the hinge loss, whereas standard SCGD
is typically analyzed for smooth outer functions. For numerical comparison,
however, we implement SCGD directly on the original non-smooth objective by using
the gradient of the outer function at its non-differentiable point, in the same
spirit as the standard SGD method used in our comparison, where the
non-differentiability at the kink point is ignored in the implementation. By doing so, we treat SCGD as a compositional subgradient baseline for comparison.

We test two classification datasets: QSAR Biodegradation and Breast Cancer Wisconsin.
The QSAR Biodegradation dataset~\cite{qsar_biodegradation_254} is a binary classification dataset from computational
chemistry, where the task is to predict whether a chemical compound is ready biodegradable
from molecular descriptor features. The Breast Cancer Wisconsin dataset~\cite{breast_cancer_wisconsin_(diagnostic)_17} is widely used for benchmarking binary classification methods in medical diagnostics, where the target label indicates whether a tumor is malignant or benign. For both datasets, we map the labels to $\{-1,+1\}$ for hinge-loss training.

Each dataset is randomly split into a 75\% train set and a 25\% test set. The features are standardized using the training set statistics. For all methods, we use mini-batches of size 64 and decaying stepsizes of the form
$\alpha_k=\alpha_0 k^{-3/4}$. For the SSG method, we additionally use the decaying smoothing
parameter $\mu_k=\mu_0 k^{-1/4}$, which is the choice for the SSG method to achieve the best
convergence rate in the convex case as discussed in Theorem~\ref{theorem: convex}. The
parameter values used in the experiments are reported in Table~\ref{tab:real-data-settings}.

\begin{table}[ht]
\centering
\begin{tabular}{lcccc}
\hline
Dataset & Epochs & Batch size & $\alpha_0$ & $\mu_0$ \\
\hline
QSAR Biodegradation & 800 & 64 & 1.2 & 4 \\
Breast Cancer Wisconsin & 150 & 64 & 1.0 & 2 \\
\hline
\end{tabular}
\caption{Experiment settings for the two datasets.}
\label{tab:real-data-settings}
\end{table}

\begin{figure}[ht]
    \centering
    \begin{minipage}{0.48\textwidth}
        \centering
        \includegraphics[width=\linewidth]{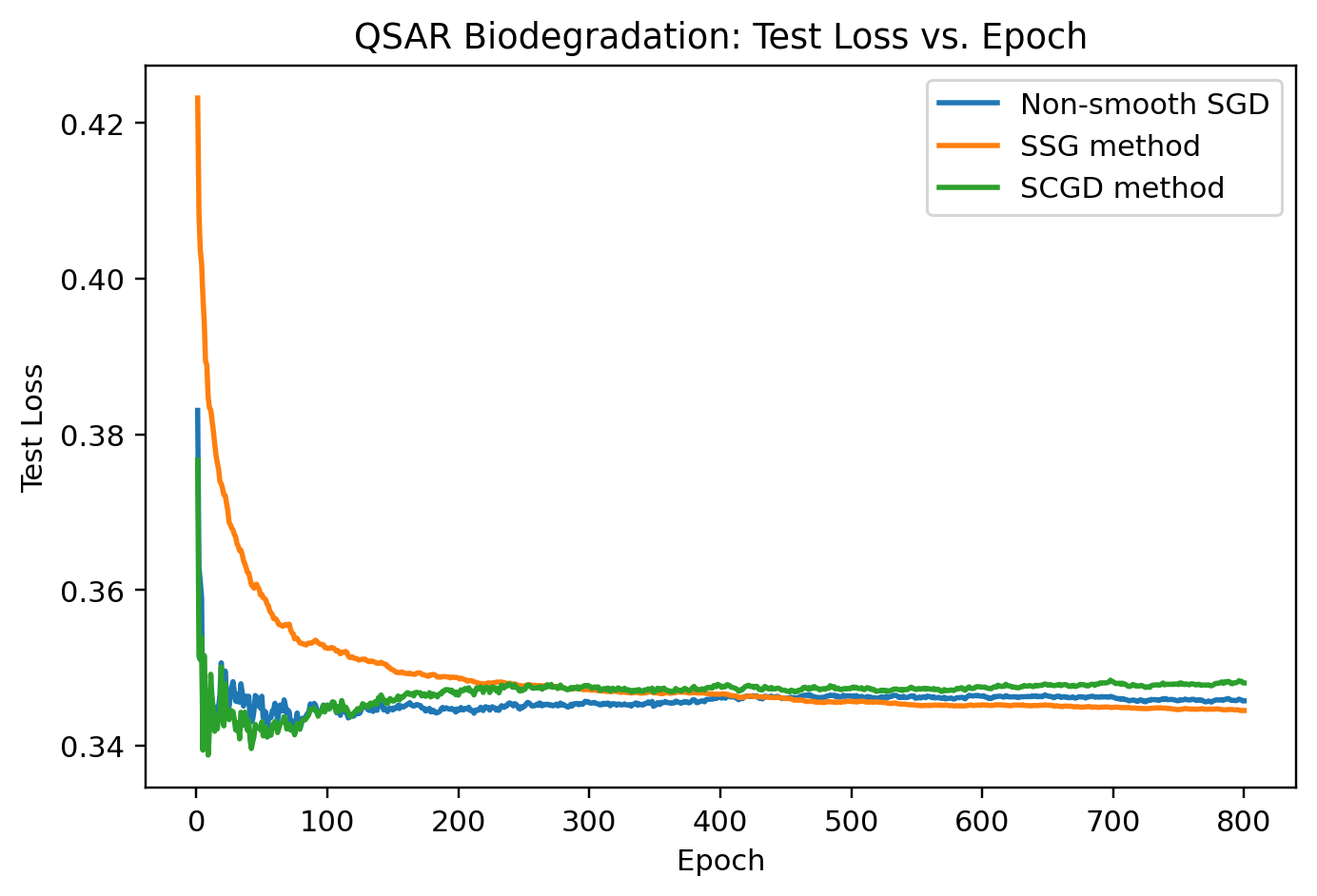}\\
        {\small (a) QSAR Biodegradation}
    \end{minipage}
    \hfill
    \begin{minipage}{0.48\textwidth}
        \centering
        \includegraphics[width=\linewidth]{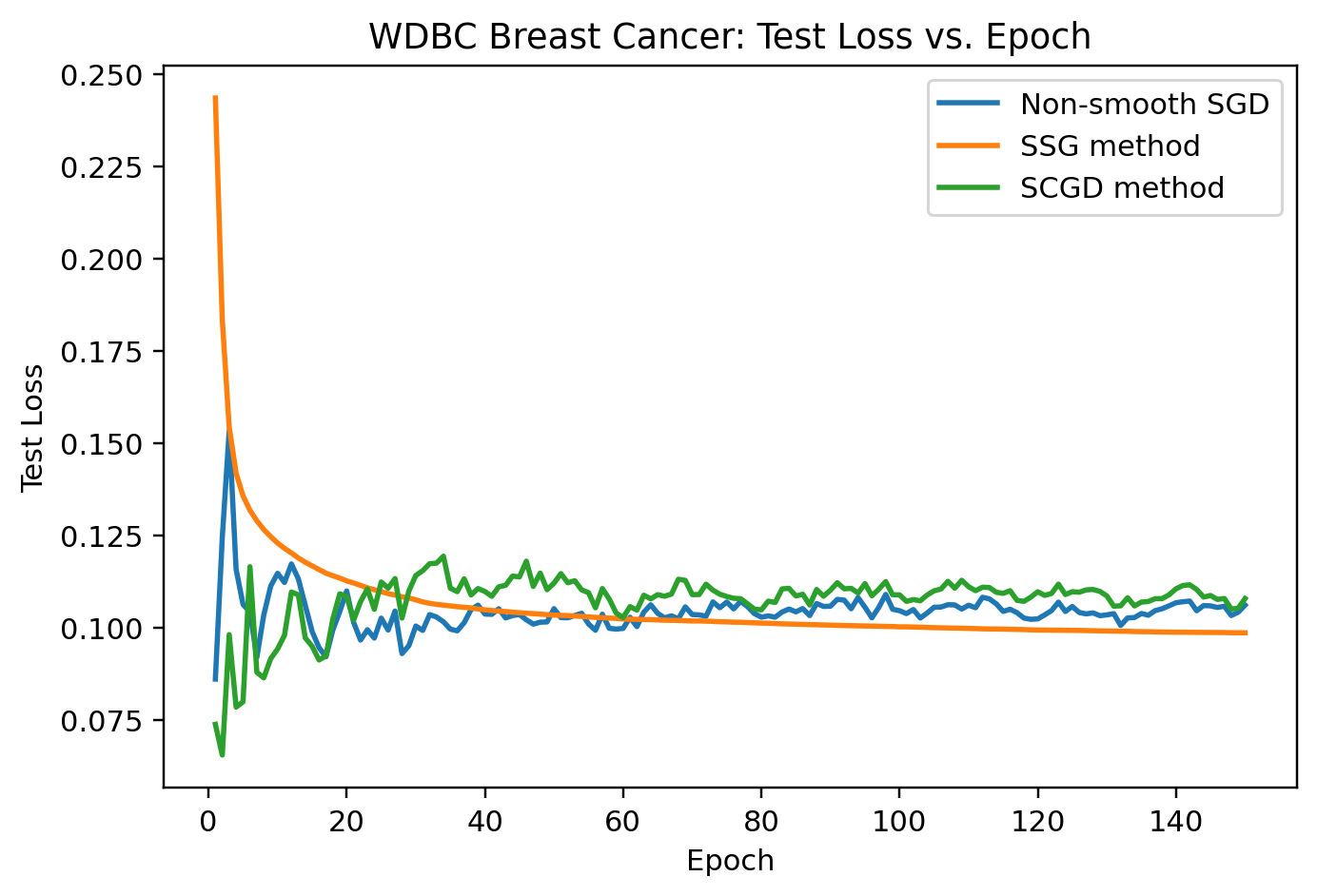}\\
        {\small (b) Breast Cancer Wisconsin}
    \end{minipage}
    \caption{Test loss curves for non-smooth SGD, SSG, and SCGD on the two classification datasets.}
    \label{fig:test-loss-real-data}
\end{figure}

Figure~\ref{fig:test-loss-real-data} plots the test-loss curves for the three methods. On both datasets, SSG achieves the lowest final test hinge loss. For the QSAR Biodegradation dataset, the SSG curve decreases more gradually at the beginning than those of the two non-smooth subgradient methods, but it eventually attains the lowest test loss. For the Breast Cancer Wisconsin dataset, SSG produces a much smoother test-loss curve and again obtains the lowest final test loss. In contrast, the non-smooth SGD and SCGD curves exhibit more visible fluctuations, which is consistent with the fact that both methods use the discontinuous subgradient of the hinge loss.

 We note that the main reason for the smoother behavior of SSG is the smoothing of the hinge loss near the
non-smooth point. For the standard non-smooth SGD and SCGD methods, samples close to the
margin can suddenly enter or leave the active set of the hinge loss, which can produce erratic
jumps in the stochastic subgradient. In contrast, the smoothing function replaces this sharp
transition by a differentiable one. As a result, samples near the margin are weighted more
gradually, leading to more stable stochastic gradients and smoother test-loss trajectories.
Moreover, since the smoothing parameter $\mu_k$ is decreased during the iterations, the
smoothed model progressively approaches the original non-smooth hinge-loss model, allowing a satisfying final test losses. 

\subsection{Results on $L_1$ loss}
We next consider the $L_1$ loss as the outer function and a linear regression model as the
inner function. The SSG method is run using the smoothing version of the absolute value
function given in~\eqref{function: |.|}. As in
the hinge-loss experiment in the previous section, both the
standard SGD and SCGD baselines are applied to the original non-smooth $L_1$ model, while
only SSG uses the smoothed approximation.

We also test two real-world regression datasets: Automobile Price and Yacht Hydrodynamics.
The Automobile Price dataset~\cite{automobile_10} contains technical and insurance-related attributes of automobiles,
and the task is to predict the price of a car. The Yacht Hydrodynamics dataset~\cite{yacht_hydrodynamics_243} contains design
parameters of sailing yachts, and the task is to predict the residuary resistance. Again, for both
datasets, the features are standardized using the training set statistics. Each dataset is randomly split into a 75\% train set and a 25\% test set. For all methods,
we use mini-batches of size 128 and decaying stepsizes of the form
$\alpha_k=\alpha_0 k^{-3/4}$ with $\alpha_0=1$. For the SSG method, we additionally use the
decaying smoothing parameter $\mu_k=\mu_0 k^{-1/4}$ with $\mu_0=5$. We run $T=800$ epochs
for both datasets.

\begin{figure}[ht]
    \centering
    \begin{minipage}{0.48\textwidth}
        \centering
        \includegraphics[width=\linewidth]{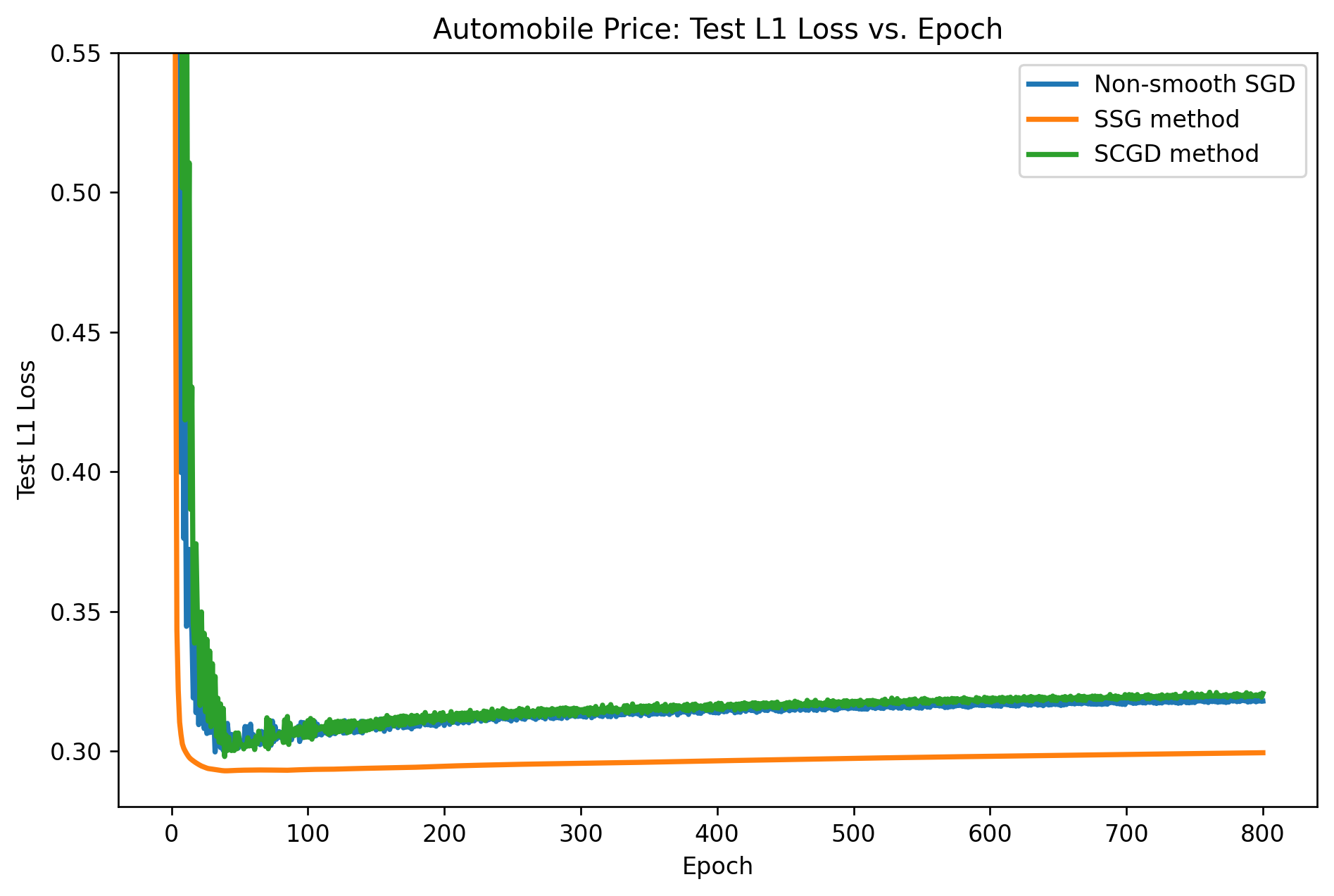}\\
        {\small (a) Automobile Price}
    \end{minipage}
    \hfill
    \begin{minipage}{0.48\textwidth}
        \centering
        \includegraphics[width=\linewidth]{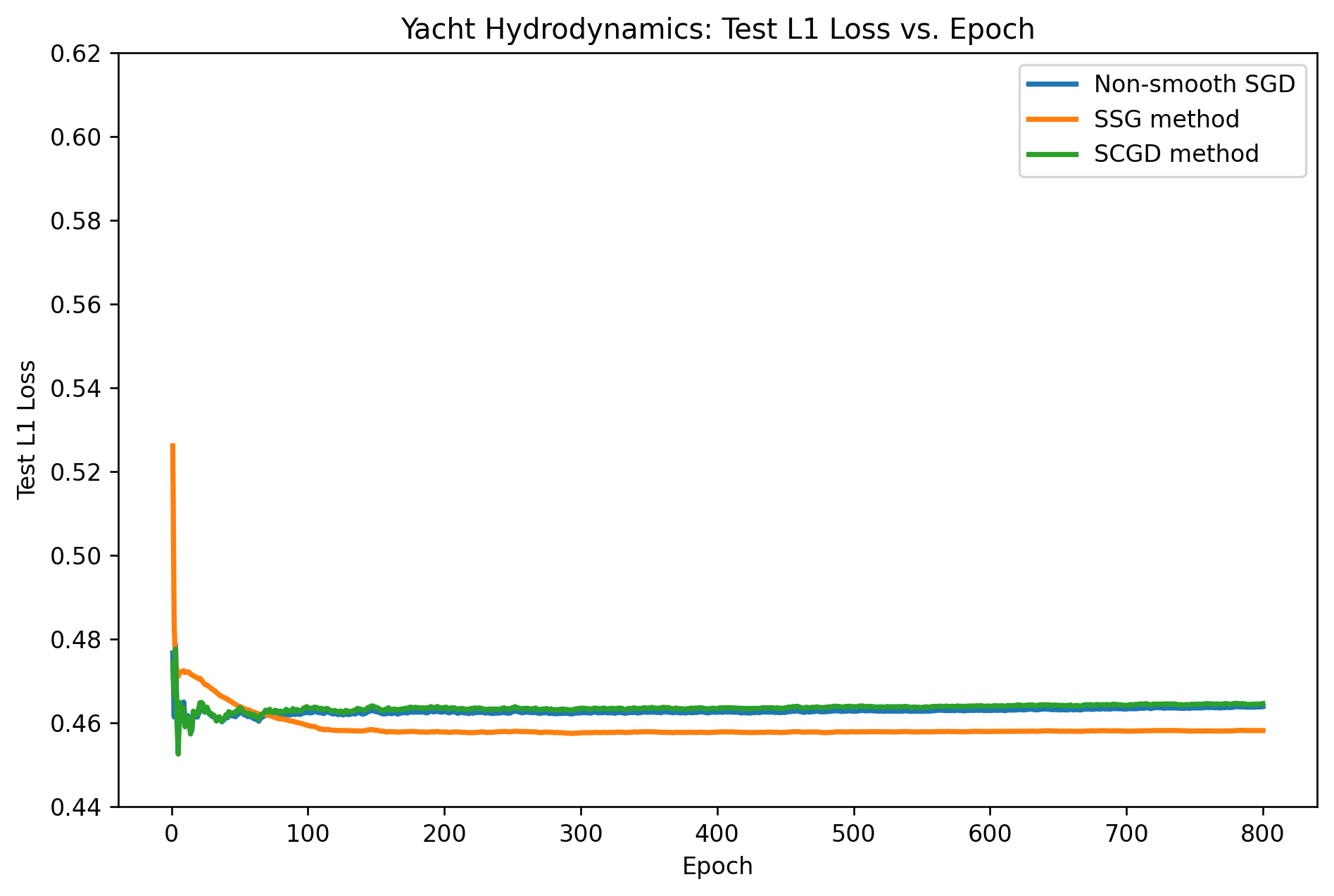}\\
        {\small (b) Yacht Hydrodynamics}
    \end{minipage}
    \caption{Test $L_1$ loss curves for non-smooth SGD, SSG, and SCGD on the two regression datasets.}
    \label{fig:l1-test-loss-real-data}
\end{figure}

Figure~\ref{fig:l1-test-loss-real-data} plots the test-loss curves for the three methods. On both datasets, SSG achieves the lowest final test $L_1$ loss. In both cases, the SSG curve is also smoother than those of non-smooth SGD and SCGD.

We observe here a behavior similar to the one in the hinge-loss case. For the standard non-smooth SGD and SCGD methods, the stochastic subgradient of $|t|$ changes abruptly when the residual crosses zero. Hence, samples with small residuals can switch their subgradient sign from one mini-batch update to the next, producing visible fluctuations in the test-loss curve. In contrast, the smoothing function replaces the kink of the $L_1$ loss with a differentiable transition region. Residuals close to zero are therefore weighted more gradually, leading to more stable stochastic gradients. Moreover, since the smoothing parameter~$\mu_k$ decreases during the iterations, the smoothed model progressively approaches the original non-smooth $L_1$ loss while still benefiting from the stabilizing effect of smoothing.

\section{Conclusions and future works}
This paper presented a smoothing-based stochastic gradient method addressing compositional optimization problems with a non-smooth outer function and a smooth inner function. The SSG method smooths the non-smooth component and progressively decreases the smoothing parameter. 
We highlight the role of the smoothing parameter, noting that iterative decreases in this parameter allow the smoothing function to more closely approximate the original non-smooth function at each iteration. Additionally, this controlled smoothing scheme prevents erratic jumps in gradient updates near non-smooth points, reduces variance in gradient estimates, and provides more consistent descent directions. Consequently, the SSG method can converge fast without significantly compromising accuracy in approximating the non-smooth problem.
We established convergence guarantees for SSG under convex, non-convex, and strongly convex settings. In the convex case, SSG matches the best-known convergence rate~$\Ocal(1/T^{1/4})$ for comparable compositional problems in which the inner function may be non-smooth. 
In the strongly convex case, the rate of SSG is arbitrarily close to~$\Ocal(1/T^{1/2})$. 
To achieve optimal convergence rates, the smoothing parameter decreases more rapidly in the strongly convex setting with $\mu_k = k^{-1/2}$, than in the convex setting with $\mu_k = k^{-1/4}$. This illustrates how stronger curvature conditions permit more aggressive reductions in the smoothing parameter, enabling faster smooth approximations. Similarly, the stepsize $\alpha_k$ must also decay more quickly ($\alpha_k = k^{-3/4}$ in the convex case and approximately $\alpha_k = k^{-3/2}$ in the strongly convex case), revealing a dependency in SSG methods, where a faster decrease in the smoothing parameter necessarily demands a correspondingly faster decay in the stepsize to preserve stability and achieve optimal convergence.

We further demonstrated how the general SSG algorithm can be specialized to two particular compositional settings of interest by satisfying the required assumptions. 
These two settings frequently arise in risk-averse optimization, machine learning, and large-scale optimization. We did not consider the double-expectation formulation examined in previous works, as the applications we target involve only a single expectation.

For future work, a promising direction is the integration of variance-reduction techniques—such as SVRG and SAGA—into this smoothing stochastic 
gradient framework, which may significantly accelerate convergence, especially in large-scale machine learning and optimization settings. Another direction is the exploration of practical applications in machine learning and risk-sensitive optimization, where compositional non-smooth structures naturally arise, such as in CVaR minimization. Evaluating the empirical performance of the SSG method in such domains remains an important task for further study.


\section*{Acknowledgments}
This work is partially supported by the U.S. Air Force Office of Scientific Research (AFOSR) award FA9550-23-1-0217 and the U.S. Office of Naval Research~(ONR) award~N000142412656.


\appendix

\section{Assumption~\ref{assumption: subgradient} for the case~$f(x)=\phi(\psi(x))$}

In this part of the appendix, we want to show strong evidence, for the case~$f(x)=\phi(\psi(x))$, that there exists~$y\in\Bcal(x,E_1\mu)$ for some~$E_1>0$, such that~$\|\nabla_x\ftilde(x,\mu)-g(y)\|\leq E_2\mu$ for some~$E_2>0$, where~$g(y)\in\partial f(y)$. In this case, we have~$\partial f(x)=\partial\phi(x)\nabla\psi(x)$,~$\ftilde(x,\mu)=\tilde{\phi}(\psi(x),\mu)$, and~$\nabla_x \ftilde(x,\mu)=\tilde{\phi}'(x,\mu)\nabla\psi(x)$.

If~$\psi(x)>\frac{\mu}{2}$, then~$\ftilde$ and~$f$ coincide, and the result is trivially satisfied for~$y=x$. If~$\psi(x)\leq\frac{\mu}{2}$, assume the existence of~$y\in\Bcal(x,E_1\mu)$ such that~$\psi(y)=0$. For most~$\phi$, such as in the $L$1 and Hinge cases (see Section~\ref{section: smoothing function}), one has~$|\tilde{\phi}'(t,\mu)|\leq E_3|t|/\mu$ for some~$E_3>0$ and~$\tilde{\phi}'(t,\mu)\in\partial\phi(0)$. Hence,~$g(y)=\tilde{\phi}'(\psi(x),\mu)\nabla\psi(x)\in\partial f(y)$, and we can derive (assuming that~$\nabla\psi$ is~$L_{\nabla\psi}$-Lipschitz continuous for some~$L_{\nabla\psi}>0$)
\begin{equation*}
    \begin{aligned}
        \|\nabla_x\ftilde(x,\mu)-g(y)\| &= |\tilde{\phi}'(\psi(x),\mu)| \|\nabla\psi(x)-\nabla\psi(y)\| \\
        &\leq \frac{E_3|\psi(x)|}{\mu} L_{\nabla\psi} \|x-y\| \\
        &\leq \frac{E_1E_3}{2} L_{\nabla\psi} \equiv E_2\mu.
    \end{aligned}
\end{equation*}

Note that for most~$\psi$, it is reasonable to assume that such a~$y$ exists. If that was not the case, then for all~$E_1>0$ and~$\psi(y)=0$, one would have~$\|y-x\|>E_1\mu$. This would imply that~$\|y-x\|>2E_1\psi(x)$, and one would obtain~$(\psi(x)-\psi(y))/\|x-y\|\leq1/(2E_2)$ for all~$E_2>0$, which would mean that~$\psi$ would be arbitrarily flat.

\section{Rate in the strongly convex case}\label{app:strongly convex}
In this part of the appendix, we analyze the rate of convergence of the SSG algorithm under the assumption that the smoothing function is strongly convex. We begin by stating the strong convexity setting and subsequently present the corresponding convergence theorem. Assumption~\ref{assumption: sc} states that the smoothing function~$\ftilde$ is strongly convex, with the strong convexity constant depending on the smoothing parameter~$\mu$.
\bassumption \label{assumption: sc}
For all $x\in\Rmbb^n$, the smoothing function $\ftilde(x, \mu)$ is strongly convex with a strong convexity constant $c(\mu)$ given by $c(\mu) = c/\mu$,
where~$c > 0$ is a constant.
\eassumption

Assumption~\ref{assumption: sc} is satisfied for the smoothing functions~\eqref{function: |.|} and~\eqref{function: Hinge} when~$x$ is relatively close to the kink point.
After having stated the assumption we need, we now establish the rate of convergence for the strongly convex case.
\btheorem [Convergence rate in the strongly convex case]
Under Assumptions~\ref{assumption: unbiasedness}--\ref{assumption: sc: diff} and~\ref{assumption: sc}, suppose that the SSG algorithm runs with the smoothing parameter~$\mu_k=k^{-a}$, where $a>0$, and stepsize~$\alpha_k=\frac{1}{(k+1)c(\mu_k)}$. Define~$f^T_{\text{best}}=\min_{k=\{0,1,\ldots,T\}}f(x_k)$. Assume that~$x_*$ is a minimizer of the function~$f$ and~$x_1$ is the initial iterate. Then for all~$s>0$ and~$a>\frac{1}{2}$, we have
\begin{equation} \label{eq: sc}
        \Embb[f^T_{\text{best}}]-f(x_*) \;\leq\; \left(\frac{4aC}{2a-1} + \frac{1}{2} \|x_1-x_*\|^2  + \frac{1}{2}G^2\right) \frac{1}{T^{1-a}} + \frac{G^2}{2s}\frac{1}{T^{1-s-a}}.
\end{equation}
\etheorem
\bproof
We first recall the update rule of the SSG algorithm. For each iteration, we set the new iterate as $x_{k+1} = x_k-\alpha_k \nabla_x \ftilde(x_k,\xi_k,\mu_k)$.
We use the update rule to get the following
\begin{equation} \label{eq: using update rule}
\begin{aligned}
     &\Embb[\|x_{k+1}-x_*\|^2|x_k] \\
     &= \Embb[\|x_k-\alpha_k \nabla_x \ftilde(x_k,\xi_k,\mu_k) - x_*\|^2|x_k] \\
     &= \Embb\|x_k-x_* \|^2|x_k] - 2 \alpha_k \Embb[\langle \nabla_x \ftilde(x_k,\xi_k,\mu_k),x_k-x_*\rangle |x_k] +\alpha_k^2\Embb[\|\nabla_x \ftilde(x_k,\xi_k,\mu_k)\|^2|x_k] \\
     &= \|x_k-x_* \|^2 - 2 \alpha_k \langle \nabla_x \ftilde(x_k,\mu_k),x_k-x_*\rangle
     +\alpha_k^2\Embb[\|\nabla_x \ftilde(x_k,\xi_k,\mu_k)\|^2|x_k].
\end{aligned}
\end{equation}
By the strongly convexity assumption of $\ftilde(x,\mu)$, we have
\begin{equation} \label{ieq: strongly convexity}
    \langle \nabla_x \ftilde(x_k,\mu_k),x_k-x_*\rangle \geq \frac{c(\mu_k)}{2}\|x_k-x_*\|^2 + \ftilde(x_k,\mu_k)-\ftilde(x_*,\mu_k).
\end{equation}
Combining \eqref{eq: using update rule} and \eqref{ieq: strongly convexity}, and taking the unconditional expectation $\Embb[\cdot]$ of both sides, the following inequality is obtained:
\begin{equation*} \label{ieq: sc: combine}
    \begin{aligned}
        \Embb[\|x_{k+1}-x_*\|^2] &\leq \Embb[\|x_k-x_* \|^2] - \alpha_k c(\mu_k) \Embb[\|x_k-x_* \|^2] \\ &\quad - 2\alpha_k\Embb[\ftilde(x_k,\mu_k)-\ftilde(x_*,\mu_k)] + \alpha_k^2 \Embb[\|\nabla_x \ftilde(x_k,\xi_k,\mu_k)\|^2].
    \end{aligned}
\end{equation*}

Then, by applying the bounded second moment assumption of the smoothing gradient as stated in Assumption \ref{assumption: bounded variance}, we have
\begin{equation} \label{ieq: sc: boundness assumption}
    \begin{aligned}
     2\alpha_k\Embb[\ftilde(x_k,\mu_k)-\ftilde(x_*,\mu_k)] &\leq (1-\alpha_kc(\mu_k)) \Embb[\|x_k-x_* \|^2] - \Embb[\|x_{k+1}-x_*\|^2]  + \frac{\alpha_k^2 }{\mu_k^2}G^2.  \\
    \end{aligned}
\end{equation}
Recall that we have $\alpha_k=\frac{1}{k^a(k+1)}$, $c(\mu_k)=k^a$, and $\mu_k=k^{-a}$. Consequently, inequality \eqref{ieq: sc: boundness assumption} can be written as:
\begin{equation*} \label{ieq: plugin ak ck}
    \begin{aligned}
     \frac{2}{k^a(k+1)}\Embb[\ftilde(x_k,\mu_k)-\ftilde(x_*,\mu_k)] &\leq \frac{k}{k+1}\Embb[\|x_k-x_* \|^2] - \Embb[\|x_{k+1}-x_*\|^2] + \frac{1}{(k+1)^2} G^2 .
    \end{aligned}
\end{equation*}
Multiplying both sides by $k+1$, it follows that
\begin{equation} \label{ieq: before sum}
    \begin{aligned}
        \frac{2}{k^a}\Embb[\ftilde(x_k,\mu_k)-\ftilde(x_*,\mu_k)] &\leq k\Embb[\|x_k-x_* \|^2] - (k+1)\Embb[\|x_{k+1}-x_*\|^2] + \frac{1}{(k+1)} G^2 .
    \end{aligned}
\end{equation}

Next, summing the inequality \eqref{ieq: before sum} from $k=1$ to $T$ for some $T>0$, we obtain
\begin{equation} \label{ieq: sum}
\begin{aligned}
    &\sum_{k=1}^T\frac{2}{k^a}\Embb[\ftilde(x_k,\mu_k)-\ftilde(x_*,\mu_k)] 
    = \sum_{k=1}^T\frac{2}{k^a}\Embb[f(x_k)-f(x_k)+f(x_*)-f(x_*)+\ftilde(x_k,\mu_k)-\ftilde(x_*,\mu_k)] \\
    &\leq \Embb[\|x_1-x_*\|^2] -  (T+1)\Embb[\|x_{T+1}-x_*\|^2] +G^2\sum_{k=1}^T \frac{1}{(k+1)} \\
    &\leq \|x_1-x_*\|^2 + G^2\sum_{k=1}^T \frac{1}{(k+1)}.
\end{aligned}
\end{equation}
From inequality \eqref{ieq: sum}, together with \eqref{eq: accuracy} (for~$x_n$ and~$x_*$), we have
\begin{equation} \label{ieq: after sum}
    \begin{aligned}
        &2\sum_{k=1}^T\frac{1}{k^a}\Embb[f(x_k)-f(x_*)]\\
        &\leq -\sum_{k=1}^T\frac{2}{k^a}\Embb[\ftilde(x_k,\mu_k)-f(x_k)] + \sum_{k=1}^T\frac{2}{k^a}\Embb[\ftilde(x_*,\mu_k)-f(x_*)]+\|x_1-x_*\|^2 + G^2\sum_{k=1}^T \frac{1}{(k+1)} \\
        &\leq 4C\sum_{k=1}^T\frac{1}{k^{2a}} + \|x_1-x_*\|^2 + G^2\sum_{k=1}^T \frac{1}{(k+1)} .
    \end{aligned}
\end{equation}

Now, our goal is to demonstrate that the right-hand side of inequality \eqref{ieq: after sum} can be bounded by a constant. To achieve this,  let us first consider the term $\sum_{k=1}^T\frac{1}{k^{2a}}$, it holds that 
\begin{equation*}
    \begin{aligned}
        \sum_{k=1}^T\frac{1}{k^{2a}} \leq 1 + \int_{1}^T \frac{1}{x^{2a}} dx = 1 + \frac{x^{1-2a}}{1-2a}\Bigg|^T_1 = \frac{T^{1-2a}-2a}{1-2a}.
    \end{aligned}
\end{equation*}
Particularly, when $a > \frac{1}{2}$, we have 
\begin{equation} \label{ieq: rt a>1/2}
    \sum_{k=1}^T\frac{1}{k^{2a}}\leq \frac{2a-T^{1-2a}}{2a-1} \leq \frac{2a}{2a-1}.
\end{equation}
Also, consider $\sum_{k=1}^T \frac{1}{(k+1)}$, it holds that
\begin{equation*} \label{ieq: st}
    \begin{aligned}
        \sum_{k=1}^T \frac{1}{(k+1)} \leq \ln T + 1.
    \end{aligned}
\end{equation*}
Notice that for all $T\geq0$ and $s>0$, the inequality $\ln T\leq \frac{T^s}{s}$ holds. 
Therefore, we have 
\begin{equation} \label{ieq: st without ln}
    \sum_{k=1}^T \frac{1}{(k+1)} \leq \ln T + 1 \leq \frac{T^s}{s} + 1.
\end{equation}

Next, by defining $f^T_{\text{best}}=\min_{k=\{0,1,\ldots,T\}}f(x_k)$ and considering again inequality \eqref{ieq: after sum}, we have
\begin{equation} \label{ieq: lhs}
    2\sum_{k=1}^T\frac{1}{k^a}\Embb[f(x_k)-f(x_*)] \geq 2\Embb[f^T_{\text{best}}-f(x_*)] \sum_{k=1}^T\frac{1}{k^a} \geq 2\Embb[f^T_{\text{best}}-f(x_*)] \frac{1}{T^{a-1}}.
\end{equation}
Thus, when $a>\frac{1}{2}$ and $s>0$, combining \eqref{ieq: after sum}, \eqref{ieq: rt a>1/2}, \eqref{ieq: st without ln}, and \eqref{ieq: lhs} gives the following:
\begin{equation*} \label{ieq: sc: conclusion}
    \begin{aligned}
        \Embb[f^T_{\text{best}}-f(x_*)] &\leq \left(\frac{4aC}{2a-1} + \frac{1}{2} \|x_1-x_*\|^2 + G^2\frac{T^s + s}{2s}\right) \frac{1}{T^{1-a}},
    \end{aligned}
\end{equation*}
from which we arrive at \eqref{eq: sc}.
\eproof

The rate in the strongly convex case is arbitrarily close to~$1/T^{1/2}$ in the sense of being worse than~$1/T^{1/2}$ but better than~$1/T^{1/p}$ for any~$p>2$. This rate can be translated into a worst-case complexity bound arbitrarily close to~$\varepsilon^{-2}$, in the sense of being slower than~$\varepsilon^{-2}$ but faster than any bound of rate $\varepsilon^{-p}$ with~$p>2$. Note that the rate in~\cite{wang2017stochastic} is~$1/T^{2/3}$, but in their approach, the authors used a similar but different problem formulation and also a form of momentum.

\section{Strongly convex case rate for~$f(x)=\phi(\Embb[\psi(x,\xi)])$}
In this part of the appendix, we consider the case~$f(x)=\phi(\Embb[\psi(x,\xi)])$ (which follows the same structure as in Section~\ref{section: compositional}) when~$f$ is strongly convex. In this setting, we aim to improve the convergence rate established for the general case of~$f$ in Appendix~\ref{app:strongly convex}. The specific algorithm we will use is shown in Algorithm~\ref{algorithm: SSG for sc}, which follows the lines of the SCGD algorithm in~\cite{wang2017stochastic}. The main modification in Algorithm~\ref{algorithm: SSG for sc} compared to Algorithm~\ref{algorithm: SSG for general compositional} in Section~\ref{section: compositional} lies in Step 2, which updates a moving average of the inner function~$\psi$ and computes the outer derivative at that averaged value. By doing so, the modified algorithm reduces the variance of~$\psi(x_k,\xi_k)$ and provides more consistent outer derivative estimates. Besides, also notice that here we will be requiring~$\{\alpha_k\}$ decaying faster than~$\{\beta_k\}$, meaning~$\alpha_k/\beta_k\to0$. 

\begin{algorithm}[H]
    \caption{SSG method for strongly convex~$f(x)=\phi(\Embb[\psi(x,\xi)])$ } \label{algorithm: SSG for sc}
\begin{algorithmic}[1]
\medskip
\item[] \textbf{Input:} $x_1$,~$y_1$,~$\{\alpha_k\}$,~$\{\beta_k\}$, and~$\{\mu_k\}$.
\medskip
\item[] \textbf{For} $k = 1,2,\dots$ \textbf{do}
\begin{enumerate}
    \item[]  {\bf Step 1.}
    Generate realizations~$\xi_k^1$ and~$\xi_k^2$ of the random variable $\xi$. Then compute $\psi(x_k,\xi_k^1)$ and~$\nabla_x \psi(x_k,\xi_k^2)$, respectively.
    \item[]  {\bf Step 2.}
    Update~$y_{k+1}=(1-\beta_k)y_k+\beta_k \psi(x_k,\xi_k^1)$ and compute~$\tilde{\phi}'(y_{k+1},\mu_k)$.
    \item[]  {\bf Step 3.} 
    Update iterate~$x_{k+1} = x_k -\alpha_k\tilde{\phi}'(y_{k+1},\mu_k)\nabla_x \psi(x_k,\xi_k^2)$.
    \item[]  {\bf Step 4.} 
    Update smoothing parameter~$\mu_{k+1} \leq \mu_k$.
\end{enumerate}
\item[] \textbf{End do}
\end{algorithmic}
\end{algorithm}

Since the problem in this appendix has the same structure as in Section~\ref{section: compositional}, we retain Assumptions~\ref{assumption: compositional: nabla_psi} and~\ref{proposition: compositional: tilde_phi}. Following~\cite{wang2017stochastic}, we also introduce the additional assumptions stated below to establish the convergence result. Assumption~\ref{Assumption: f app sc} states that the function~$f$ is strongly convex. 

\bassumption \label{Assumption: f app sc}
For all~$x\in\Rmbb^n$, the true function~$f$ is strongly convex with a strong convexity constant~$2\sigma$, where~$\sigma$ is some positive scalar.
\eassumption

In the proof of Theorem~\ref{Theorem: scimprove} below, we will use the following inequality known for strongly convex functions with minimizer~$x_*$
\begin{equation} \label{ieq: sc x*}
    f(x)-f(x_*) \;\geq\; \sigma\|x-x_*\|^2.
\end{equation}
Assumption~\ref{Assumption: phi'lip} states the Lipschitz continuity of the derivative of the smoothing outer function~$\tilde{\phi}$, where the Lipschitz constant is inversely proportional to the smoothing parameter~$\mu$.

\bassumption \label{Assumption: phi'lip}
The smoothing outer function $\tilde{\phi}$ has a Lipschitz continuous derivative with constant $L_{\tilde{\phi}}/\mu$, meaning that for all~$y,\bar{y}\in\mathbb{R}$ and all~$\mu>0$,
\begin{equation*}
    \|\tilde{\phi}'(y,\mu)-\tilde{\phi}'(\ybar,\mu)\|\;\leq\; \frac{1}{\mu}L_{\tilde{\phi}} \|y-\ybar\|,
\end{equation*}
where~$L_{\tilde{\phi}}$ is some positive scalar.
\eassumption





Assumption~\ref{Assumption:psi_unbiased_var_lipschitz} requires that the inner function~$\psi$ is Lipschitz continuous, and its stochastic sample~$\psi(\cdot,\xi)$ is an unbiased estimator with a bounded variance.

\bassumption\label{Assumption:psi_unbiased_var_lipschitz}
For all $x\in\Rmbb^{n}$, the inner function~$\psi(x)$ is Lipschitz continuous with Lipschitz constant~$G_1$ (and here we are using the same scalar as in Assumption~\ref{assumption: compositional: nabla_psi}), and its stochastic estimator~$\psi(x,\xi)$ satisfies  
\begin{enumerate}[(i)]
    \item Unbiasedness:  $\Embb[\psi(x,\xi)]=\psi(x)$.
    \item Bounded variance:  $\Embb\!\bigl[\|\psi(x,\xi)-\psi(x)\|^{2}\bigr]\le V_\psi$, where $V_\psi>0$ is some positive scalar.
\end{enumerate}
\eassumption

Assumption~\ref{Assumption: partial} states that the term~$\Embb[\tilde{\phi}'(\psi(x),\mu)\nabla_x\psi(x,\xi)]$ always lies in the subdifferential of the function~$f$.

\bassumption \label{Assumption: partial}
For all~$x\in\Rmbb^n$ and for all~$\mu>0$, it holds that 
$$\Embb[\tilde{\phi}'(\psi(x),\mu)\nabla_x \psi(x,\xi)]\;\in\;\partial f(x).$$
\eassumption

Before showing the convergence analysis of Algorithm~\ref{algorithm: SSG for sc}, we first present a lemma~(\cite[Lemma 2(a)]{wang2017stochastic}) bounding the difference between~$y_{k+1}$ and~$\psi(x_k)$ at each iteration~$k$. 

\blemma \label{lemma app c}
Let Assumptions~\ref{assumption: compositional: nabla_psi} and~\ref{Assumption:psi_unbiased_var_lipschitz} hold. Consider the sequences~$\{(x_k,y_k)\}$ generated by Algorithm~\ref{algorithm: SSG for sc}, it holds that
\begin{equation} \label{eq: app lemma}
\Embb[\|y_{k+1}-\psi(x_k)\|^2] \;\leq\; (1-\beta_k)\Embb[\|y_k - \psi(x_{k-1})\|^2] + \frac{G_1^2}{\beta_k} \Embb[\|x_k-x_{k-1}\|^2] + 2 V_\psi \beta_k^2.
\end{equation}
\elemma
\bproof
See~\cite[Supplementary Materials Section~G.1]{wang2017stochastic}.
\eproof

With the necessary assumptions and lemma listed, we now present the following convergence results. The proof of the theorem is inspired by~\cite{wang2017stochastic}.

\btheorem \label{Theorem: scimprove}
Under Assumptions~\ref{assumption: compositional: nabla_psi},~\ref{proposition: compositional: tilde_phi}, and \ref{Assumption: f app sc}--\ref{Assumption: partial}, suppose that Algorithm~\ref{algorithm: SSG for sc} runs with stepsizes~$\alpha_k=\frac{1}{k\sigma}$,~$\beta_k=\frac{1}{k^{2/3}}$, and smoothing parameter~$\mu_k=k^{-a}$, where $a>0$. Assume that~$x_*$ is the minimizer of the function~$f$. For all~$a\in(0,\frac{1}{6})$ and for sufficiently large~$T$, we have
\begin{equation*} \label{eq: sc imp}
\begin{aligned}
    \Embb[\|x_{T+1}-x_*\|^2] &\;\leq\; \Ocal\left(\frac{G_1^2G^2_2}{\sigma^2}\frac{1}{T^{1-2a}} + \frac{L^2_{\tilde{\phi}}G_1^6G_2^2}{\sigma^4}\frac{1}{T^{2/3-4a}}  + \frac{L^2_{\tilde{\phi}}G_1^2V_\psi}{\sigma^2}\frac{1}{T^{2/3-2a}} \right)    \\
    &\;=\;\Ocal\left(\frac{1}{T^{2/3-4a}}\right).
\end{aligned}
\end{equation*}
\etheorem
\begin{proof}
First, from the update rule of Algorithm~\ref{algorithm: SSG for sc}, we have
    \begin{equation} \label{ieq: appendix: xk-x}
    \begin{aligned}
        \|x_{k+1}-x_*\|^2 &= \|x_k-\alpha_k\tilde{\phi}'(y_{k+1},\mu_k)\nabla_x \psi(x_k,\xi_k^2)-x_*\|^2 \\
        &= \|x_k-x_*\|^2 - 2\alpha_k(x_k-x_*)^\T \tilde{\phi}'(y_{k+1},\mu_k)\nabla_x \psi(x_k,\xi_k^2) \\
        &\quad\;+ \alpha_k^2 \|\tilde{\phi}'(y_{k+1},\mu_k)\nabla_x \psi(x_k,\xi_k^2)\|^2 \\
        &= \|x_k-x_*\|^2 - 2\alpha_k(x_k-x_*)^\T\tilde{\phi}'(\psi(x_k),\mu_k)\nabla_x \psi(x_k,\xi_k^2) + u_k \\
        &\quad\;+ \alpha_k^2 \|\tilde{\phi}'(y_{k+1},\mu_k)\nabla_x \psi(x_k,\xi_k^2)\|^2,
    \end{aligned}
    \end{equation}
where~$u_k$ is defined as 
\begin{equation*}
    u_k = 2\alpha_k(x_k-x_*)^\T\nabla_x\psi(x_k,\xi_k^2)(\tilde{\phi}'(\psi(x_k),\mu_k)-\tilde{\phi}'(y_{k+1},\mu_k)).
\end{equation*}
To bound~$u_k$, we use Assumption~\ref{Assumption: phi'lip} and Young's inequality for products
\begin{align} \label{ieq: uk}
u_k &\leq 2\alpha_k \|x_k - x_*\| \|\nabla_x \psi(x_k, \xi_k^2)\| \|\tilde{\phi}'(\psi(x_k), \mu_k) - \tilde{\phi}'(y_{k+1}, \mu_k)\| \notag \\
&\leq \frac{2\alpha_k L_{\tilde{\phi}}}{\mu_k} \|x_k - x_*\| \|\nabla_x \psi(x_k, \xi_k^2)\| \|\psi(x_k) - y_{k+1}\| \notag \\
&\leq \alpha_k \sigma \|x_k - x_*\|^2 + \frac{\alpha_k L_{\tilde{\phi}}^2}{\sigma \mu_k^2} \|\nabla_x \psi(x_k, \xi_k^2)\|^2 \|\psi(x_k) - y_{k+1}\|^2
\end{align}


Taking the conditional expectation with respect to~$x_k$ on both sides of~\eqref{ieq: appendix: xk-x} and applying the convexity of~$f$ together with Assumption~\ref{Assumption: partial}, we obtain
\begin{equation} \label{ieq: x-x|x}
    \begin{aligned}
    \Embb[\|x_{k+1}-x_*\|^2|x_k] &= \|x_k-x_*\|^2 - 2\alpha_k(x_k-x_*)^\T\Embb[\tilde{\phi}'(\psi(x_k),\mu_k)\nabla_x \psi(x_k,\xi_k^2)|x_k]  \\
        &\quad\;+ \alpha_k^2 \Embb[\|\tilde{\phi}'(y_{k+1},\mu_k)\nabla_x \psi(x_k,\xi_k^2)\|^2|x_k]+ \Embb[u_k|x_k]\\
        &\leq \|x_k-x_*\|^2 - 2\alpha_k (f(x_k)-f^*) + \alpha_k^2 \Embb[\|\tilde{\phi}'(y_{k+1},\mu_k)\nabla_x \psi(x_k,\xi_k^2)\|^2|x_k] \\
        &\quad\;+ \Embb[u_k|x_k].
    \end{aligned}
\end{equation}
Next, by taking conditional expectation in~\eqref{ieq: uk}, plugging it into~\eqref{ieq: x-x|x}, and taking total expectation on both sides, we have
\begin{equation*}
    \begin{aligned}
        \Embb[\|x_{k+1}-x_*\|^2] &\leq (1+\sigma\alpha_k)\|x_k-x_*\|^2 - 2\alpha_k (f(x_k)-f^*) + \alpha_k^2 \Embb[\|\tilde{\phi}'(y_{k+1},\mu_k)\nabla_x \psi(x_k,\xi_k^2)\|^2] \\
        &\quad\;+ \frac{\alpha_k L_{\tilde{\phi}}^2}{\sigma\mu_k^2} \Embb[\|\nabla_x\psi(x_k,\xi_k^2)\|^2\|\psi(x_k)-y_{k+1}\|^2].
    \end{aligned}
\end{equation*} 
We now use Assumptions~\ref{assumption: compositional: nabla_psi}--\ref{proposition: compositional: tilde_phi} and~\eqref{ieq: sc x*}, and obtain
\begin{equation} \label{ieq: b2}
    \Embb[\|x_{k+1}-x_*\|^2] \leq (1-\sigma\alpha_k)\|x_k-x_*\|^2 + \alpha_k^2\left(\frac{G_1G_2}{\mu_k}\right)^2 + \frac{\alpha_k L_{\tilde{\phi}}^2G_1^2}{\sigma\mu_k^2} \Embb[\|\psi(x_k)-y_{k+1}\|^2].
\end{equation}

Now we denote~$a_k=\Embb[\|x_k-x_*\|^2]$ and~$b_k=\Embb[\|y_k-\psi(x_{k-1})\|^2]$. Let~$\Lambda_{k+1}=\max\{\frac{L_{\tilde{\phi}}^2G_1^2\alpha_k}{(\beta_k-\sigma\alpha_k)\sigma\mu_k^2}-\frac{L^2_{\tilde{\phi}}G_1^2\alpha_k}{\sigma\mu_k^2},0\}$. Given the choice of~$\alpha_k$ and~$\beta_k$ ($\alpha_k/\beta_k\to0$), one has~$\Lambda_{k+1}=\Theta(\frac{L_{\tilde{\phi}}^2G_1^2\alpha_k}{\beta_k\sigma\mu_k^2})$. Multiplying~\eqref{eq: app lemma} by~$\Lambda_{k+1}+\frac{L_{\tilde{\phi}}^2G_1^2\alpha_k}{\sigma\mu_k^2}$ and summing it with~\eqref{ieq: b2}, we obtain
\begin{equation} \label{ieq: ak+1}
    \begin{aligned}
        a_{k+1} + (\Lambda_{k+1}+\frac{L_{\tilde{\phi}}^2G_1^2\alpha_k}{\sigma\mu_k^2})b_{k+1} &\leq (1-\sigma\alpha_k)a_k+(1-\beta_k)(\Lambda_{k+1} + \frac{L^2_{\tilde{\phi}}G_1^2\alpha_k}{\sigma\mu_k^2})b_k \\
        &\quad\;+ w_k + \frac{L^2_{\tilde{\phi}}G_1^2\alpha_k}{\sigma\mu_k^2}b_{k+1},
    \end{aligned}
\end{equation}  
where
\begin{equation*}
    \begin{aligned}
        w_k &= \alpha_k^2\left(\frac{G_1G_2}{\mu_k}\right)^2 + \left(\Lambda_{k+1}+\frac{L_{\tilde{\phi}}^2G_1^2\alpha_k}{\sigma\mu_k^2} \right)\left(\frac{G_1^2}{\beta_k}\Embb[\|x_k-x_{k-1}\|^2] + 2V_\psi \beta_k^2 \right) \\
        &= \alpha_k^2\left(\frac{G_1G_2}{\mu_k}\right)^2+(\frac{L_{\tilde{\phi}}^2G_1^2}{\sigma\mu_k^2})\Ocal(\frac{\alpha^3_kG_1^4G_2^2}{\beta_k^2\mu_k^2}+V_\psi\alpha_k\beta_k).
    \end{aligned}
\end{equation*}
From~\eqref{ieq: ak+1} and the expression for~$\Lambda_{k+1}$, we can infer that
\begin{equation*}
\begin{aligned}
    a_{k+1}+\Lambda_{k+1}b_{k+1} &\leq (1-\sigma\alpha_k)a_k + (1-\beta_k)(\Lambda_{k+1} + \frac{L^2_{\tilde{\phi}}G_1^2\alpha_k}{\sigma\mu_k^2})b_k + w_k \\
     &\leq (1-\sigma\alpha_k)(a_k+\Lambda_{k+1}b_k) + w_k,   
\end{aligned}
\end{equation*}
where the second inequality comes from the expression for~$\Lambda_{k+1}$ and~$\sigma\alpha_k\leq\beta_k$.
Similarly to~\cite{wang2017stochastic}, it can be shown that~$0<\Lambda_{k+1}\leq\Lambda_k$ for sufficiently large~$k$. Thus, if~$k$ is large enough, we have
\begin{equation} \label{ieq: k suff large}
    a_{k+1} + \Lambda_{k+1}b_{k+1} \leq (1-\sigma\alpha_k)(a_k+\Lambda_kb_k) + w_k.
\end{equation}

Now, letting~$J_k =a_k+\Lambda_k b_k= \Embb[\|x_k-x_*\|^2] + \Lambda_k\Embb[\|y_k-\psi(x_{k-1})\|^2]$,~\eqref{ieq: k suff large} is rewritten as~$\Embb[J_{k+1}] \leq (1-\sigma\alpha_k)\Embb[J_k]+w_k.$
By applying~$\alpha_k=1/(\sigma k)$,~$\beta_k=k^{-2/3}$,~$\mu_k=k^{-a}$, and multiplying both sides of this inequality by~$k$, we obtain
\begin{equation} \label{ieq: app before sum}
    k\Embb[J_{k+1}] \leq (k-1)\Embb[J_k] + \frac{G_1^2G_2^2}{\sigma^2k^{1-2a}}+\Ocal\left(\frac{L^2_{\tilde{\phi}}G_1^6G_2^2}{\sigma^4k^{2/3-4a}}+\frac{L^2_{\tilde{\phi}}G_1^2V_\psi}{\sigma^2k^{2/3-2a}}\right).
\end{equation}
Next, summing the inequality~\eqref{ieq: app before sum} from~$k=1$ to~$T$ for some large~$T>0$, we obtain
\begin{equation*}
    T \Embb[J_{T+1}] \leq \Ocal \left(\frac{G_1^2G^2_2}{\sigma^2}T^{2a} + \frac{L^2_{\tilde{\phi}}G_1^6G_2^2}{\sigma^4}T^{4a+1/3}  + \frac{L^2_{\tilde{\phi}}G_1^2V_\psi}{\sigma^2}T^{2a+1/3} \right).
\end{equation*}
Thus, it follows that 
\begin{equation*}
\begin{aligned}
    \Embb[\|x_{T+1}-x_*\|^2]\leq E[J_{T+1}] &\leq \Ocal\left(\frac{G_1^2G^2_2}{\sigma^2}\frac{1}{T^{1-2a}} + \frac{L^2_{\tilde{\phi}}G_1^6G_2^2}{\sigma^4}\frac{1}{T^{2/3-4a}}  + \frac{L^2_{\tilde{\phi}}G_1^2V_\psi}{\sigma^2}\frac{1}{T^{2/3-2a}} \right)    \\
    &=\Ocal\left(\frac{1}{T^{2/3-4a}}\right).
\end{aligned}
\end{equation*}
\end{proof}
Notice that the rate in this specific strongly convex case using Algorithm~\ref{algorithm: SSG for sc} is arbitrarily close to~$1/T^{2/3}$ when~$a\to0$, which matches the rate in~\cite{wang2017stochastic} under a similar setting. 

\end{document}